\documentclass[11pt,a4paper,reqno]{elsarticle}
\usepackage{amsfonts}
\usepackage{amsmath,amssymb,amsthm,amsxtra}
\usepackage{float}
\usepackage[colorlinks, linkcolor=RoyalBlue,anchorcolor=Periwinkle,
    citecolor=Orange,urlcolor=Emerald]{hyperref}
\usepackage[usenames,dvipsnames]{xcolor}
\usepackage{enumitem}
\usepackage{geometry} 
\usepackage{graphicx}
\usepackage{subfigure}
\usepackage{bookmark}
\usepackage{tikz}\usetikzlibrary{matrix}
\usepackage{url}

\input xy \xyoption{all}

\def\graphc{JungleGreen}
\def\vertexc{blue}
\def\halfc{blue}
\def\thirdc{blue}

\theoremstyle{plain}
\newtheorem{theorem}{Theorem}[section]
\newtheorem{lemma}[theorem]{Lemma}
\newtheorem{corollary}[theorem]{Corollary}
\newtheorem{proposition}[theorem]{Proposition}
\newtheorem{conjecture}[theorem]{Conjecture}
\theoremstyle{definition}
\newtheorem{definition}[theorem]{Definition}
\newtheorem{example}[theorem]{Example}
\newtheorem{remark}[theorem]{Remark}

\newtheorem{ass}[theorem]{Assumption}
\numberwithin{equation}{section}

\def\hua{\mathcal}
\def\kong{\mathbb}
\def\<{\langle}
\def\>{\rangle}

\def\Aut{\operatorname{Aut}}
\def\Ind{\operatorname{Ind}}
\def\Sim{\operatorname{Sim}}
\def\Hom{\operatorname{Hom}}

\def\Ext{\operatorname{Ext}}
\def\Irr{\operatorname{Irr}}
\def\Stab{\operatorname{Stab}}
\def\diff{\operatorname{d}}
\def\Br{\operatorname{Br}}
\def\Rep{\operatorname{Rep}}

\def\arg{\operatorname{arg}}
\def\rank{\operatorname{rank}}
\def\Ps{\operatorname{Ps}}

\def\AR{\operatorname{AR}}

\def\dimvec{\operatorname{\underline{dim}}}

\newcommand{\h}{\operatorname{\hua{H}}}            
\newcommand{\nh}{\operatorname{\widehat{\hua{H}}}}
\renewcommand{\k}{\mathbf{k}}
\renewcommand{\mod}{\operatorname{mod}}
\newcommand{\Ho}[1]{\operatorname{\bf H}_{#1}}
\newcommand{\Wid}[1]{\operatorname{Wid}_{#1}}
\newcommand{\tilt}[3]{{#1}^{#2}_{#3}}
\newcommand{\Cone}{\operatorname{Cone}}

\def\numbers
{\begin{enumerate}[label=\arabic*{$^\circ$}.]}
\def\ends
{\end{enumerate}}

\newcommand{\EG}{\operatorname{EG}}       
\newcommand{\EGp}{\operatorname{EG}^\circ}
\newcommand{\C}[2]
{\operatorname{\hua{C}_{#1}(#2)}}               

\newcommand{\shift}[1]
{\operatorname{\Sigma_{#1}}}
\newcommand{\D}{\operatorname{\hua{D}}}

\newcommand{\qq}[1]{\operatorname{\Gamma}_{#1}Q}

\newcommand{\cub}{\operatorname{U}}

\def\zero{\hua{H}_\Gamma}
\def\nzero{\hua{H}_Q}

\def\dimvec{\operatorname{\underline{dim}}}

\newcommand\eb{edge[->,thick,>=latex,NavyBlue]}
\newcommand\es{edge[->,>=latex,dashed,thick]}
\newcommand\ee{edge[->,>=latex,thick,red]}
\newcommand\eo{edge[->,>=latex,dotted,thick]}

\newcommand\pf{\operatorname{pf}}
\newcommand\HN{\operatorname{HN}}
\newcommand\dpath{\operatorname{\overrightarrow{\mathbf{P}}}}
\newcommand\oset[1]{[#1]_{\HN}}
\newcommand\sli{\operatorname{Sli}}
\newcommand\dia{\operatorname{diam}}
\newcommand\dis{\operatorname{dis}}
\newcommand\hall{\operatorname{\mathbf{H}}}
\newcommand\DT{\operatorname{DT}}

\newtheorem{counterexample}[theorem]{Counterexample}
\newtheorem{DefLem}[theorem]{Definition/Lemma}

\newcommand{\cc}[1]{\overline{#1}}
\def\Stap{\operatorname{Stab}^\circ}
\def\HA{\hua{E}}

\def\K{\operatorname{K}}

\title{Stability conditions and quantum dilogarithm identities for Dynkin quivers}
\author{{\small Yu Qiu }}
\date{2014-9-15}
\begin{document}

\begin{abstract}
{\small

We study the fundamental groups of the exchange graphs for
the bounded derived category $\D(Q)$ of a Dynkin quiver $Q$ and
the finite-dimensional derived category $\D(\qq{N})$ of
the Calabi-Yau-$N$ Ginzburg algebra associated to $Q$.
In the case of $\D(Q)$,
we prove that its space of stability conditions (in the sense of Bridgeland) is simply connected.
As an application,
we show that the Donanldson-Thomas type invariant associated to $Q$ can be calculated
as a quantum dilogarithm function on its exchange graph.
In the case of $\D(\qq{N})$,
we show that the faithfulness of the Seidel-Thomas braid group action
(which is known for $Q$ of type $A$ or $N=2$)
implies the simply connectedness of the space of stability conditions.

\vskip .3cm
{\parindent =0pt
\it Key words:} space of stability conditions, Calabi-Yau categories,
higher cluster categories, Donaldson-Thomas invariants, quantum dilogarithm identities
}\end{abstract}

\maketitle


\section{Introduction}
\subsection{Overall}
The notion of a stability condition on a triangulated category
was defined by Bridgeland \cite{B1} (cf. \S~\ref{sec:SC}).
The idea was inspired from physics by studying D-branes in string theory.
Nevertheless, the notion itself is interesting purely mathematically.
A stability condition on a triangulated category $\D$ consists of
a collection of full additive subcategories of $\D$, known as the slicing,
and a group homomorphism from the Grothendieck group $\K(\D)$
to the complex plane, known as the central charge.
Bridgeland \cite{B1} showed a key result that
the space $\Stab(\D)$ of stability conditions on $\D$
is a finite dimensional complex manifold, provided that the rank of $\K(\D)$ is finite.
Moreover, these spaces carry interesting geometric/topological structures
which shed light on the properties of the original triangulated categories.
Most interesting examples of triangulated categories are derived categories.
They are weak homological invariants arising in both
algebraic geometry and representation theory,
and indeed different manifolds and quivers with relations
might share the same derived category (say complex projective plane and Kronecker quiver).
Also note that the space of stability conditions
are related to Kontsevich's homological mirror symmetry.
More precisely,
a quotient of the space of stability conditions of
the Fukaya category (of Lagrangian submanifolds) of a symplectic manifold
should be (conjectural) the K\"{a}hler moduli space of the mirror variety.
We will study the spaces of stability conditions of
the bounded derived category $\D(Q)$ of a Dynkin quiver $Q$ and
the finite-dimensional derived category $\D(\qq{N})$ of
the Calabi-Yau-$N$ Ginzburg algebra associated to $Q$.
Noticing that when $Q$ is of type A,
$\D(\qq{N})$ was studied by Khovanov-Seidel-Thomas \cite{KS}/\cite{ST}
via the derived Fukaya category of
Lagrangian submanifolds of the Milnor fibres of the singularities of type $A_n$.

In understanding stability conditions and triangulated categories,
t-structures play an important role.
We will always assume a t-structure is \emph{bounded}.
In fact, we can view a t-structure as a `discrete' (integer) structure
while a stability condition (resp. a slicing) is its `complex' (resp. `real') refinement.
Every t-structure carries an abelian category sitting inside it, known as its heart.
Note that an abelian category is a canonical heart in its derived category,
e.g. $\nzero=\mod\k Q$ is the canonical heart of $\D(Q)$.
The classical way to understand relations between different hearts
is via HRS-tilting (cf. \S~\ref{sec:TT}), in the sense of Happel-Reiten-Smal\o.
To give a stability condition is equivalent to giving
a t-structure and a stability function on its heart with the Harder-Narashimhan (HN) property.
This implies that a finite heart (i.e. has $n$ simples and has finite length)
corresponds to a (complex) $n$-cell in the space of stability conditions.
Moreover, Woolf \cite{W} shows that
the tilting between finite hearts corresponds to the tiling of these $n$-cells.
More precisely,
two $n$-cells meet if and only if the corresponding hearts differ by a HRS-tilting;
and they meet in codimension one if and only if the hearts differ by a simple tilting.
Following Woolf,
our main method to study a space of stability conditions of a triangulated category $\D$
is via its `skeleton' -- the exchange graph $\EG(\D)$,
that is, the oriented graphs whose vertices are hearts in $\D$
and whose edges correspond to simple (forward) tiltings between them (cf. \cite{Q}).
Figure~\ref{fig:LOGO} (taken from \cite{Q}, which in fact,
the quotient graph of $\EGp(\D(\Gamma_3 A_2))/[1]$ and $\Stap(\D(\Gamma_3 A_2))/\kong{C}$,
see \cite{S1} for more details)
demonstrates the duality between the exchange graph and
the tiling of the space of stability conditions by many cells like the shaded area,
so that each vertex in the exchange graph corresponds to a cell and
each edge corresponds to a common edge (codimension one face) of two neighboring cells.
We will prove certain simply connectedness of spaces of stability conditions.

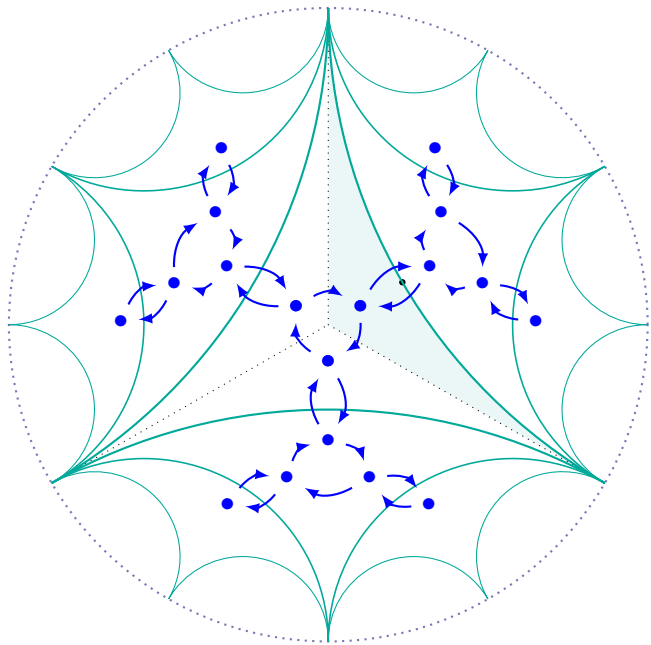
\begin{figure}[t]\centering
\begin{tikzpicture}[scale=.7]
\path (0,0) coordinate (O);
\path (0,6) coordinate (S1);
\draw[fill=\graphc!7,dotted]
    (S1) arc(360/3-360/3+180:360-360/3:10.3923cm) -- (O) -- cycle;
\draw[,dotted] (210:6cm) -- (O);
\draw[fill=black] (30:1.6077cm) circle (.05cm);
\draw
    (O)[Periwinkle,dotted,thick]   circle (6cm);
\draw
    (S1)[\graphc,thick]
    \foreach \j in {1,...,3}{arc(360/3-\j*360/3+180:360-\j*360/3:10.3923cm)} -- cycle;
\draw
    (S1)[\graphc,semithick]
    \foreach \j in {1,...,6}{arc(360/6-\j*360/6+180:360-\j*360/6:3.4641cm)} -- cycle;
\draw
    (S1)[\graphc]
    \foreach \j in {1,...,12}{arc(360/12-\j*360/12+180:360-\j*360/12:1.6077cm)}
        -- cycle;
\foreach \j in {1,...,3}
{\path (-90+120*\j:.7cm) node[\vertexc] (v\j) {$\bullet$};
 \path (-210+120*\j:.7cm) node[\vertexc] (w\j) {$\bullet$};
 \path (-90+120*\j:2.2cm) node[\vertexc] (a\j) {$\bullet$};
 \path (-90+15+120*\j:3cm) node[\vertexc] (b\j) {$\bullet$};
 \path (-90-15+120*\j:3cm) node[\vertexc] (c\j) {$\bullet$};}
\foreach \j in {1,...,3}
{\path[->,>=latex] (v\j) edge[\thirdc,bend left,thick] (w\j);
 \path[->,>=latex] (a\j) edge[\thirdc,bend left,thick] (b\j);
 \path[->,>=latex] (b\j) edge[\thirdc,bend left,thick] (c\j);
 \path[->,>=latex] (c\j) edge[\thirdc,bend left,thick] (a\j);}
\foreach \j in {1,...,3}
{\path (60*\j*2-1-60:3.9cm) node[\vertexc] (x1\j) {$\bullet$};
 \path (60*\j*2+1-120:3.9cm) node[\vertexc] (x2\j) {$\bullet$};
}
\foreach \j in {1,...,3}
{\path[->,>=latex] (v\j) edge[\halfc,bend left,thick] (a\j);
 \path[->,>=latex] (a\j) edge[\halfc,bend left,thick] (v\j);
 \path[->,>=latex] (b\j) edge[\halfc,bend left,thick] (x1\j);
 \path[->,>=latex] (c\j) edge[\halfc,bend left,thick] (x2\j);
 \path[->,>=latex] (x1\j) edge[\halfc,bend left,thick] (b\j);
 \path[->,>=latex] (x2\j) edge[\halfc,bend left,thick] (c\j);}
\end{tikzpicture}
\caption{Exchange graphs as the skeleton of space of stability conditions}\label{fig:LOGO}
\end{figure}

Stability conditions naturally link to Donaldson-Thomas (DT) invariants,
which were originally defined as the weighted Euler characteristics (by Behrend function)
of moduli spaces for Calabi-Yau 3-folds (cf. \cite{N1}).
Reineke \cite{R} (cf. \S~\ref{sec:DT.R})
realized that the DT-type invariant for a Dynkin quiver
can be calculated as a product of quantum dilogarithms, indexing by any HN-stratum of $\nzero$,
which is a `maximal refined version' of torsion pairs on an abelian category.
His approach was integrating certain identities in Hall algebras
to show the stratum-independence of the product.
We will give a combinatorial proof of these type of quantum dilogarithm identities
via exchange graphs.

\subsection{Contents}
We will collect related background in \S~\ref{sec:Preliminaries}.

In \S~\ref{sec:embed},
we prove a general result (Theorem~\ref{thm:c.e.s}) that,
under Assumption~\ref{ass},
the exchange graph $\EG_0$ can be embedded into the corresponding (connected component of)
the space $\Stab_0$ of stability conditions with a surjection
$\pi_1(\EG_0)\twoheadrightarrow\pi_1(\Stab_0)$.

In \S~\ref{sec:sc.Q},
we first make a key observation (Proposition~\ref{pp:45}) that
the fundamental groups of the exchange graphs of $\EG(Q)$
is generated by squares and pentagons.
Moreover, we prove (Theorem~\ref{thm:sc1}) that
the simply connectedness of the space of stability conditions on $\D(Q)$.

In \S~\ref{sec:sc.CY.Q}, we will (Corollary~\ref{cor:CYstab})
identify a principal component $\Stap(\qq{N})$ of $\Stab(\D(\qq{N}))$,
for any $N\geq2$, and show that (Corollary~\ref{cor:conn})
the faithfulness of the Seidel-Thomas braid group action
(which is known for $Q$ of type $A$ or $N=2$)
implies the simply connectedness of $\Stap(\qq{N})$.
Further, the quotient space $\Stap(\qq{N})/\Br$
is the `right' space of stability conditions for the higher cluster category $\C{N-1}{Q}$
(cf. Remark~\ref{rem:stab}).
In fact, the generators of its fundamental group
provide a topological realization of almost completed cluster tilting objects in $\C{N-1}{Q}$
(Theorem~\ref{thm:sc2}).

In \S~\ref{sec:sc limit},
we present (Theorem~\ref{thm:limit}) a limit formula of spaces of stability conditions
\[\Stab(Q) \cong \lim_{N\to\infty} \Stap(\qq{N})/\Br(\qq{N}),\]
which reflects a philosophical point of view that, in a suitable sense,
\begin{gather}\label{eq:philo}
    Q = \lim_{N\to\infty}\qq{N}.
\end{gather}

In \S~\ref{sec:DPHN}, we study directed paths in exchange graphs,
which naturally corresponds to Keller's green mutation (see, e.g. \cite{Q3}).
We will first show (Theorem~\ref{thm:HN}) that
HN-strata of $\nzero$ can be naturally interpreted as
directed paths connecting $\nzero$ and $\nzero[1]$ in $\EG(Q)$.
Then we discuss total stability of stability functions (cf. Conjecture~\ref{conj:R})
and the path-inducing problem.
We will provide explicit examples and a conjecture.

In \S~\ref{sec:QD via EG},
we observe that the existence of DT-type invariant of $Q$
is equivalent to the path-independence of the quantum dilogarithm product
over certain directed paths.
Then we give a slight generalization (Theorem~\ref{thm:DT}) of this path-independence,
to all paths (not necessarily directed) whose vertices lie between $\nzero$ and $\nzero[1]$.
The point is that this path-independence reduces to the cases
of squares and pentagons in Proposition~\ref{pp:45}.
Therefore such type of quantum dilogarithm identities
is just certain composition of the classical Pentagon Identities.
We will also discuss the wall-crossing formula for APR-tilting (cf. \cite{N}).
Note that Keller \cite{K6} also spotted this phenomenon
and proved some more remarkable quantum dilogarithm identities
via mutation of quivers with potential (cf. \cite{K10}).
In fact, his formula can also be rephrased as quantum dilogarithm product over paths
in the exchange graph of the corresponding Calabi-Yau-$3$ categories.

\subsection*{Acknowledgements}
This work is part of my Ph.D thesis,
under the supervision of Alastair King and supported by a U.K. Overseas Research Studentship.
I would like to thank him for the patient guidance throughout my Ph.D period.
I would also like to thank Bernhard Keller for enlightening conversations on Isle of Skye.
Final thanks to Jon Woolf for sharing his expertise on topology,
Chris Brav for clarifying the faithfulness of braid group actions
and an anonymous referee for pointing out numerous typos.

\section{Preliminaries}\label{sec:Preliminaries}
\subsection{Dynkin Quivers}\label{sec:quiver.D}
A \emph{(simply laced) Dynkin quiver} $Q=(Q_0,Q_1)$ is
a quiver whose underlying unoriented graph is one of the following unoriented graphs:
\begin{gather}\label{eq:labeling}
\begin{array}{llr}
    \xymatrix{A_{n}:} \quad &
    \xymatrix{
        1 \ar@{-}[r]& 2 \ar@{-}[r]& \cdots \ar@{-}[r]& n }         \\
    \xymatrix@R=0.1pc{\\D_{n}:} \quad &
    \xymatrix@R=0.1pc{
        &&&& n-1 \\
        1 \ar@{-}[r]& 2 \ar@{-}[r]& \cdots \ar@{-}[r]& n-2 \ar@{-}[dr]\ar@{-}[ur]\\
        &&&& n &}                                                             \\
    \xymatrix{E_{6,7,8}:} \quad &
    \xymatrix@R=1pc{
        && 4 \\
        1 \ar@{-}[r]& 2 \ar@{-}[r]& 3 \ar@{-}[r]\ar@{-}[u]& 5 \ar@{-}[r]& 6
                \ar@{-}[r]& 7 \ar@{-}[r]& 8}
\end{array}
\end{gather}
For a Dynkin quiver $Q$, we denote by $\k Q$ the \emph{path algebra};
denote by $\mod\k Q$ the \emph{category of finite dimensional $\k Q$-modules},
which can be identified with $\Rep_{\k}(Q)$,
the \emph{category of representations} of $Q$ (cf. \cite{ASS1}).
We will not distinguish between $\mod \k Q$ and $\Rep_{\k}(Q)$.
Recall that the \emph{Euler form}
\[
    \<-,-\>:\kong{Z}^{Q_0}\times\kong{Z}^{Q_0}\to\kong{Z}
\]
associated to the quiver $Q$ is defined by
\[
    \<\mathbf{\alpha}, \mathbf{\beta}\>
    =\sum_{i\in Q_0}\alpha_i\beta_i-\sum_{(i\to j)\in Q_1}\alpha_i\beta_j.
\]
Moreover for $M,L\in\mod \k Q$, we have
\begin{gather}\label{eq:euler form}
    \<\dim M,\dim L\>=\dim\Hom(M,L)-\dim\Ext^1(M,L),
\end{gather}
where $\dim E\in\kong{N}^{Q_0}$ is the \emph{dimension vector} of any $E\in\mod \k Q$.

\subsection{Hearts and t-structures}\label{sec:DC}
Let $\hua{D}(Q)=\hua{D}^b(\mod \k Q)$ be the \emph{bounded derived category} of $Q$,
which is a triangulated category.
Recall (e.g. from \cite{B1}) that a \emph{t-structure}
on a triangulated category $\hua{D}$ is
a full subcategory $\hua{P} \subset \hua{D}$
with $\hua{P}[1] \subset \hua{P}$ and such that, if one defines
\[
  \hua{P}^{\perp}=\{ G\in\hua{D}: \Hom_{\hua{D}}(F,G)=0,
  \forall F\in\hua{P}  \},
\]
then, for every object $E\in\hua{D}$, there is
a unique triangle $F \to E \to G\to F[1]$ in $\hua{D}$
with $F\in\hua{P}$ and $G\in\hua{P}^{\perp}$.
Any t-structure is closed under sums and summands
and hence it is determined by the indecomposables in it.
Note also that $\hua{P}^{\perp}[-1]\subset \hua{P}^{\perp}$.

A t-structure $\hua{P}$ is \emph{bounded} if for every object $M$,
the shifts $M[k]$ are in $\hua{P}$ for $k\gg0$ and in $\hua{P}^{\perp}$ for $k\ll0$.
The \emph{heart} of a t-structure $\hua{P}$ is the full subcategory
\[
  \h=  \hua{P}^\perp[1]\cap\hua{P}
\]
and any bounded t-structure is determined by its heart.
More precisely, any bounded t-structure $\hua{P}$
with heart $\h$ determines, for each $M$ in $\hua{D}$,
a canonical filtration (\cite[Lemma~3.2]{B1})
\begin{equation}
\label{eq:canonfilt}
\xymatrix@C=0,5pc{
  0=M_0 \ar[rr] && M_1 \ar[dl] \ar[rr] &&  ... \ar[rr] && M_{m-1}
        \ar[rr] && M_m=M \ar[dl] \\
  & H_1[k_1] \ar@{-->}[ul]  && && && H_m[k_m] \ar@{-->}[ul]
  }
\end{equation}
where $H_i \in \h$ and $k_1 > ... > k_m$ are integers.
Using this filtration,
one can define the $k$-th homology of $M$, with respect to $\h$,
to be
\begin{gather}\label{eq:homology}
 \Ho{k}(M)=
 \begin{cases}
   H_i & \text{if $k=k_i$} \\
   0 & \text{otherwise.}
 \end{cases}
\end{gather}
Then $\hua{P}$ consists of those objects
with no negative homology and $\hua{P}^\perp$
those with only negative homology.
For any (non-zero) object $M$ in $\hua{D}$, define the
\emph{(homological) width} $\Wid{\h}(M)$
to be the difference $k_1-k_m$ of the maximum and minimum degrees
of its non-zero homology as in \eqref{eq:canonfilt}.
It is clear that the width is invariant under shifts.

In this paper, we only consider bounded t-structures and their hearts
and all categories will be implicitly assumed to be $\k$-linear.
For two hearts $\h_1$ and $\h_2$ in $\hua{D}$
with corresponding t-structure $\hua{P}_1$ and $\hua{P}_2$,
we say $\h_1 \leq \h_2$ if and only if $\hua{P}_1\supset\hua{P}_2$,
or equivalently, $\hua{P}^\perp_1\subset\hua{P}^\perp_2$,
with equality if and only if the two hearts are the same.

Note that a heart is always abelian.
For instance, $\D(Q)$ has a canonical heart $\mod \k Q$,
which we will write as $\nzero$ from now on.
Recall an object in an abelian category is \emph{simple}
if it has no proper subobjects, or equivalently
it is not the middle term of any (non-trivial) short exact sequence.
We will denote a complete set of simples of an abelian category $\hua{C}$ by
$\Sim\hua{C}$.
Denote by $\<T_1, ..., T_m\>$ the smallest full
subcategory containing $T_1,...,T_m$ and closed under extensions.

\subsection{Sections in AR-quiver}
For quivers, a convenient way to picture the categories
$\nzero$ and $\hua{D}(Q)$
is by drawing their Auslander-Reiten (AR) quivers.

\begin{definition}\cite[Chapters II,IV]{ASS1}
The \emph{AR-quiver} $\AR (\hua{C})$
of a ($\k$-linear) category $\hua{C}$ is defined as follows.
\begin{itemize}
\item
    Its vertices are identified with elements of $\Ind\hua{C}$,
    a complete set of indecomposables of $\hua{C}$,
    i.e. a choice of one indecomposable object in each isomorphism class.
\item
    The arrows from $X$ to $Y$ are identified with
    a basis of $\Irr(X,Y)$,
    the space of irreducible maps $X\to Y$ (see \cite[IV~1.4~Definition]{ASS1}).
\item
    There is a (maybe partially defined) bijection, called AR-translation,
\[
   \tau: \Ind\hua{C} \to \Ind\hua{C},
\]
    with the property that there is an arrow from $X$ to $Y$
    if and only if there is a corresponding arrow from $\tau Y$ to $X$.
    Moreover, we have the AR-formula
    \[
    \Ext^1(Y,X)\cong \Hom(X,\tau Y)^*.
    \]
\end{itemize}
\end{definition}

For example, here is (part of) the AR-quiver of $\hua{D}(Q)$ for $Q$ of type $A_4$.
The black vertices are the AR-quiver of
$\nzero$, when $Q$ has a straight orientation.
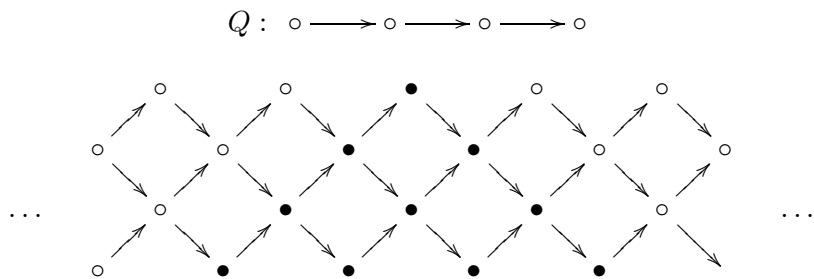
\[Q:\,\xymatrix{\circ \ar[r]&\circ \ar[r]&\circ \ar[r]&\circ}\]
\[
 \xymatrix@R=1pc@C=1pc{
 && \circ \ar[dr] && \circ \ar[dr] && \bullet \ar[dr]
    && \circ \ar[dr] && \circ \ar[dr] && \\
&   \circ \ar[ur]\ar[dr] && \circ \ar[ur]\ar[dr] &&
    \bullet \ar[ur]\ar[dr] && \bullet \ar[ur]\ar[dr] &&
    \circ \ar[ur]\ar[dr] && \circ \\
\dots && \circ \ar[dr]\ar[ur] &&
    \bullet \ar[dr]\ar[ur] && \bullet \ar[dr]\ar[ur]
    && \bullet \ar[dr]\ar[ur] && \circ \ar[dr]\ar[ur] && \dots\\
&   \circ \ar[ur] && \bullet \ar[ur] && \bullet \ar[ur] &&
    \bullet \ar[ur] && \bullet \ar[ur] &&\\
}
\]

Recall that, the \emph{translation quiver} $\kong{Z}Q$ of an acyclic quiver $Q$
is the quiver whose vertex set is $\kong{Z}\times Q_0$
and whose arrow set is $\{a_m, b_m \mid a\in Q_1, m\in\kong{Z}\}$,
where $a_m$ is the arrow $(m,i)\to(m,j)$ and $b_m$ is the arrow $(m,i)\leftarrow(m-1,j)$
for any arrow $a:i\to j$ in $Q_1$.
Further, a section $\Sigma$ in $\kong{Z}Q$ is a full subquiver satisfying the following
\begin{itemize}
\item   $\Sigma$ is acyclic.
\item   $\Sigma$ meets the vertex set $\kong{Z}\times \{i\}$ exactly once,
for any $i\in Q_0$.
\item   If the head and tail of a path $p$ in $\kong{Z}Q$ are in $\Sigma$,
then any vertex in $p$ is in $\Sigma$.
\end{itemize}
When $Q$ is of Dynkin type, $\AR(\D(Q))$ is isomorphic to $\kong{Z}Q$.
In particular, we have \cite{H}
\begin{gather}\label{eq:ar}
\Ind\D(Q)=\bigcup_{j\in\kong{Z}}\Ind\nzero[j].
\end{gather}

We will give several characterization of standard hearts in $\D(Q)$
in this subsection.
Following \cite[Chapter IX]{ASS1}, we introduce several notions.
\begin{definition}
    A \emph{path} $p$ in $\AR(\hua{C})$ is a sequence
    $$\xymatrix@C=0.5cm{
      M_0 \ar[r]^{f_1} & M_1 \ar[r]^{f_2} & M_2 \ar[r] &
      ... \ar[r] & M_{t-1} \ar[r]^{f_t} & M_t }$$
    of irreducible maps $f_i$
    between indecomposable modules $M_i$ with $t \geq 1$.
    When such a path exists, we say that $M_0$ is a \emph{predecessor} of $M_t$
    or $M_t$ is a \emph{successor} of $M_0$.
    A path $p$ is called \emph{sectional} if,
    for all $1 < i \leq t$, $\tau M_i \ncong M_{i-2}$.
    Denote by $\Ps(M)$ the set of objects that lie in some sectional path starting from $M$
    and by $\Ps^{-1}(M)$ be the set of objects that lie in some sectional path ending at $M$.
\end{definition}

We have the following elementary lemma.
\begin{lemma}\cite{ASS1}
\label{lem:basic}
Let $\Delta$ be the underlying graph of a Dynkin quiver $Q$.
Any section in $\kong{Z}Q$ is isomorphic to some orientation of $\Delta$.
In particular,
the projectives of $\nzero$
together with the irreducible maps between them form a section in $\AR(\D(Q))$,
with exactly the opposite orientation of $Q$.
Further, for any object $M$ in $\kong{Z}Q$,
$\Ps(M)$ and $\Ps^{-1}(M)$ form two sections.
\end{lemma}

For a section $P$ in $\AR(\D(Q))\cong\kong{Z}Q$,
define $[P,\infty)= \bigcup_{m \geq 0} \tau^{-m} P$.
Similarly for $(-\infty,P]$ and define
$[P_1,P_2]=[P_1,\infty)\cap(-\infty,P_2]$.
The following lemmas characterize such type of intervals.

\begin{lemma}
\label{lem:ps}
The interval $[\Ps(M),\infty)$ consists precisely of all the successors of $M$.
Similarly,
$(-\infty,\Ps^{-1}(M)]$ consists precisely all the predecessors of $M$.
\end{lemma}
\begin{proof}
We only prove the first assertion. The second is similar.

By the local property of the translation quiver $\kong{Z}Q$,
any object in $[\Ps(M),\infty)$ is a successor of $M$.
On the other hand, let $L$ be any successor of $M$ with path
\[
  M=M_0 \xrightarrow{f_1} M_1 \to \dots \xrightarrow{f_j} M_j=L.
\]
If $\tau M_i =M_{i-2}$ for some $2 \leq i \leq j$,
then consider $\tau L$ with path
\[
    M=M_0 \xrightarrow{f_1} \dots \xrightarrow{f_{i-2}} M_{i-2}=\tau M_i
    \xrightarrow{\tau f_i}  \tau M_{i+1} \xrightarrow{\tau f_{i+2}} \dots
    \xrightarrow{\tau{f_{j}}} \tau M_j= \tau L.
\]
we can repeat the process until the path is sectional,
i.e. until we obtain $\tau^k L \in \Ps(M)$ for some $k\geq 0$.
Thus $L\in [\Ps(M),\infty)$.
\end{proof}

\begin{lemma}
\label{lem:homs}
Let $M,L \in \Ind\hua{D}(Q)$.
If $\Hom(M,L) \neq 0$ then
\begin{gather*}
    L \in \Big[\Ps(M),\Ps^{-1} \big( \tau(M[1]) \big) \Big],\quad
    M \in \Big[ \Ps \big( \tau^{-1}(L[-1]) \big), \Ps^{-1}(L) \Big].
\end{gather*}
\end{lemma}
\begin{proof}
By the Auslander-Reiten formula,
we have
\[
    \Hom( M,L )^* \,=\, \Hom( \tau^{-1}(L),M[1]  ).
\]
The lemma now follows from Lemma~\ref{lem:ps}.
\end{proof}

For later use, we define the position function as follows.
\begin{DefLem}\label{def:pf}
There is a \emph{position function}
\[\pf:\AR(\D(Q))\to\kong{Z},\]
unique up to an additive constant, such that
$\pf(M)-\pf(\tau M)=2$ for any $M\in\AR(\D(Q))$
and $\pf(M)<\pf(L)$ for any successor $L$ of $M$.
For a heart $\h$ in $\D(Q)$, define \[\pf(\h)=\sum_{S\in\Sim\h}\pf(S).\]
\end{DefLem}

\subsection{Standard hearts in $\D(Q)$}\label{sec:SH}
\begin{proposition}
\label{pp:standard}
Let $Q$ be a Dynkin quiver.
A section $P$ in $\AR(\hua{D}(Q))$
will induce a unique t-structure $\hua{P}$ on $\hua{D}(Q)$
such that $\Ind \hua{P} = [P,\infty)$.
For any t-structure $\hua{P}$ on $\D(Q)$, the following are equivalent:
\numbers
\item
    $\hua{P}$ is induced by some section $P$.
\item
    $\Ind\hua{D}(Q) = \Ind\hua{P} \cup \Ind\hua{P}^\perp$.
\item
    The corresponding heart $\h$ is isomorphic to $\h_{Q'}$,
    where $Q'$ has the same underlying diagram of $Q$.
\item
    $\Wid{\h}M = 0$ for any $M \in \Ind\hua{D}(Q)$,
    where $\h$ is the corresponding heart.
\ends
\end{proposition}

\begin{proof}
For a section $P$,
let $\hua{P}$ be the subcategory
which is generated by the elements in $\Ind \hua{P} = [P,\infty)$.
Notice that $\Ind\hua{P}^\perp = (\infty, \tau^{-1}P]$,
which implies $\hua{P}$ is a t-structure.
Thus $1^\circ\Rightarrow2^\circ$.
Since $\h=[P, P[1])$, $1^\circ\Rightarrow3^\circ$.

If $\h$ is isomorphic to $\h_{Q'}$ for some quiver $Q'$,
then $\Ind\hua{P}=\cup_{j \geq 0} \h[j]= [ P',\infty ) $,
where $P'$ is the sub-quiver in $\AR(\hua{D}(Q))$ consists of the projectives.
Thus $3^\circ$$\Rightarrow$$1^\circ$.

Since for any $M \in \Ind\hua{D}(Q)$,
$\Wid{\h}M =0$ if and only if $M \in \h[k]$ for some integer $k$,
we have $3^\circ\Rightarrow4^\circ$.
Noticing that $\h[k]$ is either in $\hua{P}$ or $\hua{P}^\perp$,
we have $4^\circ\Rightarrow2^\circ$.

Now we only need to prove $2^\circ\Rightarrow1^\circ$.
If an indecomposable $M$ is in $\hua{P}$ (resp. $\hua{P}^\perp$),
then, inductively, all of its successors (resp. predecessors) are in
$\hua{P}$ (resp. $\hua{P}^\perp$).
By the local property, $\tau^m(M)$ is a successor of
$M$ if $m>0$ and a predecessor if $m<0$.
Hence, in any row that contains $v$ in $\kong{Z}Q\cong\AR(\hua{D}(Q))$,
for any vertex $v\in Q_0$,
there is a unique integer $m_v$ such that
$\tau^j (v) \in \hua{P}$, for $j\geq m_v$, while
$\tau^j (v) \in \hua{P}^\perp$, for $j<m_v.$
Furthermore, for a neighboring vertex $w$ of $v$,
the local picture looks like this
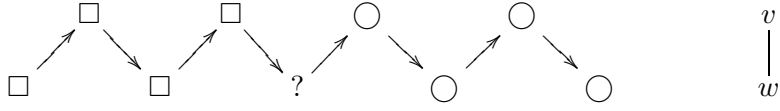
\[
\xymatrix@R=1pc@C=1pc{
&  \square \ar[dr] && \square \ar[dr] && \bigcirc \ar[dr] &&
    \bigcirc    \ar[dr]   &&&&  v \ar@{-}[d]\\
\square  \ar[ur]   &&   \square \ar[ur] && ? \ar[ur] &&
    \bigcirc \ar[ur]   &&   \bigcirc &&& w
}
\]
where $\bigcirc \in \hua{P}$ and $\square \in \hua{P}^\perp$.
Hence $v_{m_v}$ and $w_{m_w}$ must be connected in $\kong{Z}Q$ and
so the full sub-quiver of $\kong{Z}Q$
consisting of all vertices $\{v_{m_v}\}_{v\in Q_0}$
is a section and furthermore it induces $\hua{P}$.
\end{proof}

We call a heart on $\hua{D}(Q)$ is \emph{standard} if
the corresponding t-structure is induced by a section.

\subsection{Exchange graphs}\label{sec:TT}

A similar notion to a t-structure on a triangulated category
is a torsion pair on an abelian category.
Tilting with respect to a torsion pair in the heart of a t-structure
provides a way to pass between different t-structures.

A \emph{torsion pair} in an abelian category $\hua{C}$ is a pair of
full subcategories $\<\hua{F},\hua{T}\>$ of $\hua{C}$,
such that $\Hom(\hua{T},\hua{F})=0$ and furthermore
every object $E \in \hua{C}$ fits into a short exact sequence
$ \xymatrix@C=0.5cm{0 \ar[r] & T \ar[r] & E \ar[r] & F \ar[r] & 0}$
for some objects $T \in \hua{T}$ and $F \in \hua{F}$.
Here $\hua{T}$ is the torsion part of the torsion pair and $\hua{F}$ is the torsion-free part.
We will use the notation $\h=\<\hua{F},\hua{T}\>$ to indicate
an abelian category with a torsion pair.

By \cite{HRS},
for any heart $\h$ (in a triangulated category)
with torsion pair $\<\hua{F},\hua{T}\>$,
there exists the following two hearts with torsion pairs
\[
    \h^\sharp=\<\hua{T},\hua{F}[1]\>,\quad \h^\flat=\<\hua{T}[-1],\hua{F}\>.
\]
We call $\h^\sharp$ the \emph{forward tilt} of $\h$
with respect to the torsion pair $\<\hua{F},\hua{T}\>$,
and $\h^\flat$ the \emph{backward tilt} of $\h$.
Clearly $\hua{T}=\h\cap\h^\sharp$,$\hua{F}=\h\cap\h^\sharp[-1]$ and $\h^\flat=\h^\sharp[-1]$.

The following lemma collect several well-known facts about tilting.
\begin{lemma}[\cite{HRS}, cf. also \cite{Q}]\label{lem:well known}
Let $\h$ be a heart in a triangulated category $\D$.
There are canonical bijections between
\begin{itemize}
\item the set of torsion pairs in $\h$,
\item the set of hearts between $\h[-1]$ and $\h$,
\item the set of hearts between $\h$ and $\h[1]$,
\end{itemize}
which sends a torsion pair to its backward and forward tilts respectively.
\end{lemma}

We say a forward tilting is \emph{simple},
if the corresponding torsion free part is generated by a single simple object $S$,
and denote the heart by $\tilt{\h}{\sharp}{S}$.
Similarly, a backward tilting is \emph{simple}
if the corresponding torsion part is generated by a single simple object $S$,
and denote the heart by $\tilt{\h}{\flat}{S}$.

\begin{definition}
Define the \emph{exchange graph} $\EG(\D)$ of a triangulated category $\D$
to be the oriented graph
whose vertices are all hearts in $\D$
and whose edges correspond to simple forward tiltings between them.
\end{definition}
We will label an edge of $\EG(\D)$ by the simple object of the corresponding tilting,
i.e. the edge with end points $\h$ and $\tilt{\h}{\sharp}{S}$
will be labeled by $S$.

\begin{definition}\label{def:line}
For $S \in \Sim\h$, inductively define
\[
    \tilt{\h}{m\sharp }{S}
    ={  \Big( \tilt{\h}{ (m-1) \sharp}{S} \Big)  }^{\sharp}_{S[m-1]}
\]
for $m\geq1$ and similarly for $\tilt{\h}{m\flat }{S}, m\geq1$.
We will write $\tilt{\h}{m\sharp }{S}=\tilt{\h}{-m\flat }{S}$ for $m<0$.
A \emph{line} $l=l(\h, S)$ in $\EG(\D)$,
for some triangulated category $\D$, is the full subgraph
consisting of the vertices $\{ \tilt{\h}{m\sharp}{S} \}_{m\in\kong{Z}}$,
for some heart $\h$ and a simple $S\in\Sim\h$.
We say an edge in $\EG(\D)$ has \emph{direction}-$T$
if its label is $T[m]$ for some integer $m$;
we say a line $l$ has \emph{direction}-$T$
if some (and hence every) edge in $l$ has direction-$T$.
\end{definition}

By \cite{KV}, we know that $\EG(\D(Q))$ is connected when $Q$ is of Dynkin type,
which will be written as $\EG(Q)$.
For an alternate proof, see Appendix~\ref{app}.

\subsection{Calabi-Yau categories and braid groups}\label{sec:CY}
Let $N>1$ be an integer.
Denote by $\qq{N}$ the \emph{Calabi-Yau-$N$ Ginzburg (dg) algebra} associated to $Q$,
which is constructed as follows (\cite[\S~7.2]{K10}, \cite{G}):
\begin{itemize}
\item   Let $Q^N$ be the graded quiver whose vertex set is $Q_0$
and whose arrows are: the arrows in $Q$ with degree $0$;
an arrow $a^*:j\to i$ with degree $2-N$ for each arrow $a:i\to j$ in $Q$;
a loop $e^*:i\to i$ with degree $1-N$ for each vertex $e$ in $Q$.
\item   The underlying graded algebra of $\qq{N}$ is the completion of
the graded path algebra $\k Q^N$ in the category of graded vector spaces
with respect to the ideal generated by the arrows of $Q^N$.
\item   The differential of $\qq{N}$ is the unique continuous linear endomorphism homogeneous
of degree $1$ which satisfies the Leibniz rule and
takes the following values on the arrows of $Q^N$:
\[
    \diff a^*=0,\qquad
    \diff \sum_{e\in Q_0} e^*=\sum_{a\in Q_1} \, [a,a^*] .
\]
\end{itemize}
Write $\D(\qq{N})$ for $\hua{D}_{fd}(\mod  \qq{N})$,
the \emph{finite dimensional derived category} of $\qq{N}$
(cf. \cite[\S~7.3]{K10}).

By \cite{K4} (see also \cite{KS},\cite{ST}), we know that
$\hua{D}(\qq{N})$ is a Calabi-Yau-$N$ category
which admits a canonical heart $\zero$ generated
by simple $\qq{N}$-modules $S_e, e\in Q_0$.
Denote by $\EGp(\qq{N})$ the principal component of the exchange graph
$\EG(\D(\qq{N}))$, that is, the component containing $\zero$.

We recall that each simple in $\Sim\zero$ is an $N$-spherical object (cf. \cite{KS});
moreover, every such spherical object $S$ in $\hua{D}(\qq{N})$
induces an auto-equivalence in $\Aut\hua{D}(\qq{N})$, known as
the \emph{twist functor $\phi_S$} of $S$.
Denote by $\Br(\qq{N})$, the \emph{Seidel-Thomas braid group},
that is, the subgroup of $\Aut\hua{D}(\qq{N})$ generated by
the twist functors of the simples in $\Sim\zero$.

\subsection{Inducing hearts}
Recall some notation and results from \cite[\S~7.3]{Q}.
There is a special kind of exact functors from $\D(Q)$ to $\D(\qq{N})$,
known as the \emph{Lagrangian immersions} (L-immersions), see \cite[Definition~7.3]{Q}.
Let $\h$ be a heart in $\D(\qq{N})$ with $\Sim\h=\{ S_1,...,S_n \}$.
If there is a L-immersion
\[\hua{L}\colon\hua{D}(Q) \to \D(\qq{N})\]
and a heart $\nh\in\EG(Q)$ with $\Sim\nh=\{ \widehat{S}_1,...,\widehat{S}_n \}$,
such that $\hua{L}(\widehat{S}_i)=S_i$,
then we say that $\h$ is induced from $\nh$ via $\hua{L}$
and write $\hua{L}_*(\nh)=\h$.
Furthermore, let $\h$ be a heart in some exchange graph $\EGp(\qq{N})$.
Define the interval $\EG_N(\qq{N}, \h)$ to be the full subgraph of $\EGp(\qq{N})$ induced by
\[
    \{\h_0 \mid \h\in\EGp(\qq{N}),
        \h\leq\h_0\leq\h[N-2] \}
\]
and $\EGp_N(\qq{N}, \h)$ its principal component
(that is, the connected component containing $\h$).
Similarly for $\EG_N(Q,\nh)$ and $\EGp_N(Q,\nh)$.

A maximal \emph{line segment} $l$ in $\EGp_N(\qq{N},\zero)$ is a (simple forward) tilting sequence
\begin{gather}\label{eq:line}
    l\colon\hua{H}\xrightarrow{\,S\,}\tilt{\h}{\sharp }{S}\xrightarrow{S[1]}\cdots\xrightarrow{S[N-3]}
    \tilt{\h}{(N-2)\sharp }{S}
\end{gather}
The \emph{cyclic completion} $\overline{\EGp_N}(\Gamma_N Q,\zero)$ of $\EGp_N(\Gamma_N Q,\zero)$
is the oriented graph obtained from $\EGp_N(\Gamma_N Q,\zero)$ by adding an edge
$\h\to\tilt{\h}{(N-2)\sharp }{S}$ for each maximal line segment \eqref{eq:line}.
Similarly, we can define the cyclic completion of $\EGp_N(Q,\h_Q)$.

\begin{theorem}\cite[Theorems~8.1 and 8.6]{Q}\label{thm:Q}
For an acyclic quiver $Q$, we have the following:
\numbers
\item
    there is a canonical L-immersion $\hua{I}:\D(Q)\to\D(\qq{N})$ that induces an isomorphism
    \begin{gather}\label{eq:I1}
        \hua{I}_*\colon\EGp_N(Q,\nzero) \to \EGp_N(\qq{N},\zero).
    \end{gather}
\item
    as graphs, we have
    \begin{gather}\label{eq:eg}
        \overline{\EGp_N}(Q,\nzero)\cong\overline{\EGp_N}(\qq{N}, \zero)\cong\EGp(\qq{N})/\Br;
    \end{gather}
\item
    for any heart $\h$ in $\EGp(\qq{N})$,
    $\Sim\h$ has $n$ elements and $\{\phi_S\}_{S\in\Sim\h}$ is a generating set for $\Br(\qq{N})$;
\item
    for any line $l$ in $\EGp(\qq{N})$,
    its orbit in $\EGp(\qq{N})/\Br$ is a $(N-1)$ cycle.
\ends
\end{theorem}

Besides, we have
\begin{proposition}\label{pp:lim}
Let $Q$ be a Dynkin quiver.
$\EG_{N}(Q,\nzero)$ is finite for any $N>1$ and we have
\begin{gather}\label{eq:lim eg}
    \EG(Q)=\lim_{N\to\infty} \EG_{2N}(Q,\nzero[1-N]).
\end{gather}
\end{proposition}
\begin{proof}
Notice that there are only finitely many indecomposables in
$\bigcup_{j=0}^{N-2}\nzero[j]$ and
hence only finitely many hearts in $\EGp_N(Q,\nzero)$.

Let $\nh\in\EG(Q)$.
For any simple $\widehat{S}$ of $\nh$,
by considering its homology $\Ho{\bullet}$ with respect to $\nzero$,
we have $\widehat{S}\in\bigcup_{j=1-N}^{N-1}\nzero[j]$ provided $N\gg1$.
This implies that $\nzero[-N+1]\leq\nh\leq\nzero[N-1]$.
Then $\nh\in\EG_{2N}(Q,\nzero[1-N])$, which implies \eqref{eq:lim eg}.
\end{proof}

\subsection{Stability conditions}\label{sec:SC}
This section (following \cite{B1})
collects the basic definitions of stability conditions.
Denote $\hua{D}$ a triangulated category and
$\K(\hua{D})$ its Grothendieck group.

\begin{definition}\cite[Definition~5.1]{B1}\label{def:stab}
A \emph{stability condition} $\sigma = (Z,\hua{P})$ on $\hua{D}$ consists of
a group homomorphism $Z:\K(\hua{D}) \to \kong{C}$ called the \emph{central charge} and
full additive subcategories $\hua{P}(\varphi) \subset \hua{D}$
for each $\varphi \in \kong{R}$, satisfying the following axioms:
\numbers
\item if $0 \neq E \in \hua{P}(\varphi)$
then $Z(E) = m(E) \exp(\varphi  \pi \mathbf{i} )$ for some $m(E) \in \kong{R}_{>0}$,
\item for all
$\varphi \in \kong{R}$, $\hua{P}(\varphi+1)=\hua{P}(\varphi)[1]$,
\item if $\varphi_1>\varphi_2$ and $A_i \in \hua{P}(\varphi_i)$
then $\Hom_{\hua{D}}(A_1,A_2)=0$,
\item for each nonzero object $E \in \hua{D}$ there is a finite sequence of real numbers
$$\varphi_1 > \varphi_2 > ... > \varphi_m$$
and a collection of triangles
$$\xymatrix@C=0.8pc@R=1.4pc{
  0=E_0 \ar[rr] && E_1 \ar[dl] \ar[rr] &&   E_2 \ar[dl] \ar[rr] && ... \
  \ar[rr] && E_{m-1} \ar[rr] && E_m=E \ar[dl] \\
  & A_1 \ar@{-->}[ul]  && A_2 \ar@{-->}[ul] &&  && && A_m \ar@{-->}[ul]
},$$
with $A_j \in \hua{P}(\varphi_j)$ for all j.
\ends
\end{definition}

We call the collection of subcategories $\{\hua{P}(\varphi)\}$,
satisfying $2^\circ \sim 4^\circ$ in Definition~\ref{def:stab},
the \emph{slicing} and
the collection of triangles in $4^\circ$ the \emph{Harder-Narashimhan (HN) filtration}.
For any nonzero object $E \in \hua{D}$ with HN-filtration above,
define its upper phase to be $\Psi^+_{\hua{P}}(E)=\varphi_1$
and lower phase to be $\Psi^-_{\hua{P}}(E)=\varphi_m$.
By \cite[Lemma~5.2]{B1} , $\hua{P}(\varphi)$ is abelian.
A (non-zero) object $E\in\hua{P}(\varphi)$ for some $\varphi\in\kong{R}$
is said to be semistable;
in which case, $\varphi=\Psi^{\pm}_{\hua{P}}(E)$.
Moreover, if $E$ is simple in $\hua{P}(\varphi)$,
then it is said to be stable.
Let $I$ be an interval in $\kong{R}$ and define
\[
    \hua{P}(I)=\{ E\in\hua{D} \mid  \Psi^{\pm}_{\hua{P}}(E) \in I \}\cup\{0\}.
\]
Then for any $\varphi\in \kong{R}$,
$\hua{P}[\varphi,\infty)$ and $\hua{P}(\varphi,\infty)$
are t-structures in $\hua{D}$.
Further, we say the heart of a stability condition $\sigma = (Z,\hua{P})$ on $\hua{D}$
is $\hua{P}[0,1)$.
For any heart $\h\in\D$,
let $\cub(\h)$ be the set of stability conditions in $\D$ whose heart is $\h$.

There is a natural $\kong{C}$ action
on the set $\Stab(\D)$ of all stability conditions on $\D$, namely:
\[
    \Theta \cdot (Z,\hua{P})=(Z \cdot z,\hua{P}_x),
\]
where $z=\exp(\Theta \pi \mathbf{i}), \Theta=x+y \mathbf{i}$
and $\hua{P}_x(m)=\hua{P}(x+m)$ for $x,y,m \in \kong{R}$.
There is also a natural action on $\Stab(\D)$ induced by $\Aut(\D)$, namely:
$$\xi \circ (Z,\hua{P})=(Z \circ \xi,\xi \circ \hua{P}).$$

Similarly to stability condition on triangulated categories,
we have the notion of stability function on abelian categories.

\begin{definition}\cite{B1}\label{def:sf}
A \emph{stability function} on an abelian category $\hua{C}$
is a group homomorphism $Z:\K(\hua{C})\to\kong{C}$ such that
for any object $0\neq M\in\hua{C}$, we have
$Z(M) = m(M) \exp(\mu_Z(M) \mathbf{i} \pi)$ for some
$m(M) \in \kong{R}_{>0}$ and $\mu_Z(M)\in [0,1)$,
i.e. $Z(M)$ lies in the upper half-plane
\begin{gather}\label{eq:H}
    H=\{r\exp(\mathrm{i} \pi \theta)\mid r\in\kong{R}_{>0},0\leq \theta<1 \}
    \subset\kong{C}.
\end{gather}
Call $\mu_Z(M)$ the phase of $M$.
We say an object $0\neq M\in\hua{C}$ is semistable (with respect to $Z$) if every
subobject $0\neq L$ of $M$ satisfies $\mu_Z(L) \leq \mu_Z(M)$.
Further, we say a stability function $Z$ on $\hua{C}$ satisfies HN-property,
if for an object $0\neq M\in\hua{C}$,
there is a collection of short exact sequences
\begin{equation*}
\xymatrix@C=1pc@R=1.4pc{
  0=M_0 \ar@{^{(}->}[rr] && M_1 \ar@{->>}[dl] \ar@{^{(}->}[rr] &&  ...
  \ar@{^{(}->}[rr] && M_{k-1} \ar@{^{(}->}[rr] && M_k=M \ar@{->>}[dl] \\
  & L_1 && && && L_k
}\end{equation*}
in $\hua{C}$ such that $L_1,...,L_k$ are semistable objects
(with respect to $Z$) and their phases are in decreasing order,
i.e. $\phi(L_1)>\cdots>\phi(L_k)$.
\end{definition}

Note that we have a different convention $0\leq \theta<1$
for the upper half plane $H$ in \eqref{eq:H} as Bridgeland's $0<\theta\leq1$.

Then we have another way to give a stability condition on triangulated categories.

\begin{proposition}\cite{B1}\label{pp:ss}
To give a stability condition on a triangulated category $\hua{D}$
is equivalent to giving a bounded t-structure on $\hua{D}$ and
a stability function on its heart with the HN-property.
Further,
to give a stability condition on $\hua{D}$ with a finite heart $\h$
is equivalent to giving a function $\Sim\h \to H$,
where $H$ is the upper half plane as in \eqref{eq:H}.
\end{proposition}

Recall a crucial result of spaces of stability conditions.

\begin{theorem}\label{thm:B1}\cite[Theorem~1.2]{B1}
All stability conditions on a triangulated category
$\D$ form a complex manifold, denoted by $\Stab(\D)$;
each connected component of $\Stab(\D)$ is locally homeomorphic to a linear sub-manifold of
$\Hom_{\kong{Z}}(\K(\D),\kong{C})$, sending a stability condition $(\h, Z)$ to its central change $Z$.
\end{theorem}

Therefore every finite heart $\h$ corresponds to
a (complex, half closed and half open) $n$-cell
$\cub(\h)\simeq H^n$ inside $\Stab(\D)$ (cf. \cite[Lemma~5.2]{B2}).

\section{The canonical embedding}\label{sec:embed}
In this section, we discuss the relationship
between exchange graphs and stability conditions in general.
Let $\EG_0$ be a connected component of the exchange graph $\EG(\D)$
of some triangulated category $\D$
and $\Stab_0(\D)=\bigcup_{\h\in\EG_0}\cub(\h)$.
We will impose the following finiteness condition
throughout this section.

\begin{ass}\label{ass}
Every heart in $\EG_0$ is finite and has only finitely many torsion pairs.
\end{ass}

We first recall a result of Woolf, which describes
certain connected components of space of stability conditions and
how stability conditions in an $n$-cell $\cub(\h)$ degenerate in such cases.
Note that Assumption~\ref{ass} is slightly weaker then the original \cite[Assumption~2]{W},
but is sufficient for the following theorem (cf. \cite[Remark~2.15]{W}).

\begin{theorem}\cite[Proposition~2.17 and Theorem~2.18]{W}\label{thm:woolf}
Under Assumption~\ref{ass}, $\Stab_0(\D)$
is a connected component in the space of stability conditions of $\D$.
Moreover, for any $\h\in\EG_0$, a stability condition $\sigma=(Z,\hua{P})$
in the boundary of $\cub(\h)$ is determined by allowing the central charge of
a stability condition in the interior of $\cub(\h)$ to degenerate in such a way that
the central charges of a non-empty set of simples become real.
The only constraint on this degeneration is that
there is no $M\in\h$ with $Z(M)=0$ which is semi-stable for each of
a sequence of stability conditions in $\cub(\h)$ with central charges $Z_i\to Z$.
\end{theorem}

Thus, we have
$\cc{\cub(\h)}-\cub(\h)=\bigcup_{\h[-1]\leq\h'<\h} \left( \cc{\cub(\h)}\cap\cub(\h') \right)$
and hence the gluing structure of
\[\Stab_0=\bigcup_{\h\in\EG_0}\cc{\cub(\h)}\]
is encoded by the following formula
\begin{gather}
\label{eq:boundary}
    \partial\cub(\h) =\bigcup_{\h[-1]\leq\h'<\h}
        \left( \cc{\cub(\h)}\cap\cub(\h') \right)
    \cup\bigcup_{\h<\h'\leq\h[1]}
        \left( \cc{\cub(\h')}\cap\cub(\h) \right).
\end{gather}
Call a term in the RHS in \eqref{eq:boundary} a face of the $n$-cell $\cub(\h)$.
We precede to give a more careful discussion about the gluing structure of $\Stab_0$
via its `skeleton' $\EG_0$.

\begin{lemma}\label{lem:top}
Suppose that there is a retraction $R$ from $X\times[0,1)$ to $Y$,
for some open contractible space $X$ and $Y\subset X\times[0,1)$,
such that the restriction of $R$ on $X\times\{0\}$ is a contraction $C$ from $X$ to a point $y$.
For any contraction $C'$ from $X$ to $y'\in X$ and any $0<\epsilon<1$,
there is a retraction $R'$ from $X\times[0,1)$ to $Y'$, such that
$Y'|_{X\times[\epsilon,1)}=Y$ and the retraction $R'$ restricted to $X\times\{0\}$ is $C'$.
\end{lemma}
\begin{proof}
The retraction $R$ can be viewed as a subspace of $\left(X\times[0,1)\right)\times[0,1]$
and the contraction $C$ is the intersection of $R$ and $X\times\{0\}\times[0,1]$.
Let $F$ be a homotopy from the contraction $C$ to $C'$.
We can view $F$ as a subspace of $X\times[-1,0]\times[0,1]$,
where $F$ intersect $X\times\{0\}\times[0,1]$ is the contraction $C$
and $F$ intersect $X\times\{-1\}\times[0,1]$ is the contraction $C'$.
Gluing $R$ and $F$ along $C$ (cf. Figure~\ref{fig:top}), we get another retraction $R''$,
from $X\times[-1,1)$ to some subspace.
Clearly there is homeomorphism $h:X\times[-1,1)\cong X\times[0,1)$
such that $h|_{X\times(0,\epsilon]}$ is an isomorphism.
Then $h$ induces the required retraction $R'$ of $X\times[0,1)$
from the retraction $R''$ of $X\times[-1,1)$.

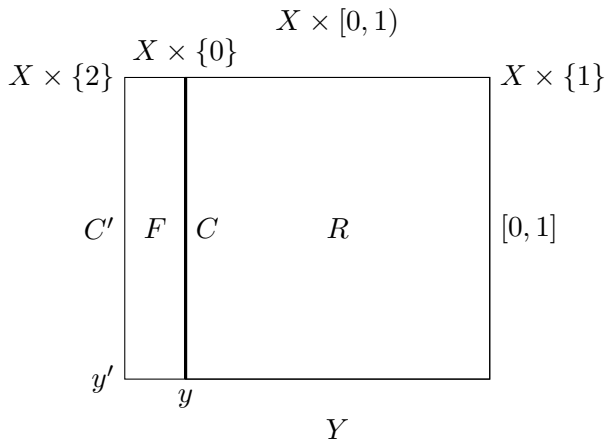
\begin{figure}[ht]\centering
\begin{tikzpicture}[scale=4]
\draw (0,0) rectangle (1,1);\draw (1,0) rectangle (-.2,1);
\draw[line width=.1pc] (0,1) -- (0,0);
\draw (1,.5) node[right] {$[0,1]$};
\draw (1,1) node[right] {$X\times\{1\}$};
\draw (-.2,1) node[left] {$X\times\{2\}$};
\draw (0,1) node[above] {$X\times\{0\}$};
\draw (.5,1.1) node[above] {$X\times [0,1)$};
\draw (.5,-.1) node[below] {$Y$};
\draw (.5,.5) node {$R$};\draw (-.1,.5) node {$F$};
\draw (0,.5) node[right] {$C$};
\draw (-.2,.5) node[left] {$C'$};
\draw (0,0) node[below] {$y$};
\draw (-.2,0) node[left] {$y'$};
\end{tikzpicture}
\caption{The retraction $R''$}
\label{fig:top}
\end{figure}
\end{proof}

\begin{theorem}\label{thm:c.e.s}
Let $\EG_0$ be a connected component of the exchange graph $\EG(\D)$
satisfying Assumption~\ref{ass}.
Then there is a canonical embedding $\iota:\EG_0 \hookrightarrow \Stab_0$,
unique up to homotopy, such that
\begin{itemize}
\item   $\iota(\h)\in(\cub(\h))^\circ$ for any heart $\h\in\EG_0$.
\item   $\iota(s)$ is contained in $(\cub(\h)\cup\cub(\tilt{\h}{\sharp}{S}))^\circ$
and transversally intersects $(\cub(\h)\cap\cc{\cub(\tilt{\h}{\sharp}{S})})^\circ$
at exactly one point,
for any edge $s:\h\to\tilt{\h}{\sharp}{S}$.
\end{itemize}
Further, there is a surjection $\pi_1(\EG_0)\twoheadrightarrow\pi_1(\Stab_0)$.
\end{theorem}
\begin{proof}
Let \[Y=\bigcup^{\h\in\EG_0}_{F_{>1}\subset\cub(\h)} F_{>1},\]
where the union is over faces $F_{>1}$ with codimension strictly greater than one.
Since $\Stab_0$ is a manifold,
$\pi_1(\Stab_0)$ can be computed using smooth loops and smooth homotopies between them
(cf. e.g. \cite[Theorem 3.8.16]{C}).
Since the decomposition of $\Stab_0$ into $n$ cells (and their faces)
is locally-finite by \eqref{eq:boundary},
we can choose a small neighbourhood $W$ of any loop $l$ which meets only finitely many faces.
By \cite[Proposition~12.4]{He}
we can perturb the loop $l$ by arbitrarily small homotopy (remaining within $W$),
such that the resulting representative loop for the original homotopy class lies in $\Stab_0-Y$.
Therefore, there is a surjection
\begin{gather}\label{eq:surj}
    \pi_1(\Stab_0-Y)\twoheadrightarrow\pi_1(\Stab_0).
\end{gather}
Next, we show that $\Stab_0-Y$ retracts to $\EG_0$.
To do so, it is sufficient to show the following (local) statements:
\numbers
\item For any heart $\h$, $\cub(\h)$ has $2n$ codimension one faces,
which correspond to the $2n$ simple tilts of $\h$,
where $n=\rank \K(\D)$.
\item For any codimension one face $F$ of some $\cub(\h)$,
$F^\circ$ is contractible to any chosen point in $F^\circ$.
\item For any heart $\h$ and
any fixed contractions for the $2n$ codimension one faces of $\cub(\h)$,
\begin{gather}\label{eq:retract}
    \cc{\cub(\h)}\cap\left(\Stab_0-Y\right)
\end{gather}
locally retracts to $\EG_0$.
That is, \eqref{eq:retract} retracts to the star $\mathbf{S}_{2n}$
such that the restrictions to the codimension one faces of the retraction
are the fixed contractions.
\ends
Note that $3^\circ$ ensures that we can glue the local retractions to a global one
and the star $\mathbf{S}_{m}$ is a tree with one internal node and $m$ leaves.

Every heart $\h$ in $\EG(\h)$ is finite by \cite[Theorem~5.7]{Q},
so $1^\circ$ is precisely \cite[Lemma~5.5]{B2}.
Moreover, recall that
$\cub(\h)$ is isomorphic to $H^n$, where $H$ is the upper half plane, as in \eqref{eq:H}.
So we have an isomorphism
\[
    I_{\h}\colon\cub(\h) \xrightarrow{\cong} \kong{R}^n\times [0,1)^n.
\]
By Theorem~\ref{thm:woolf},
any codimension one face of $\cub(\h)$ is isomorphic to $\kong{R}^n\times [0,1)^{n-1}$,
which implies $2^\circ$.
Then the boundary of \eqref{eq:retract} consists of the $2n$ interiors of
the codimension one faces of $\cub(\h)$.
Further, under the isomorphism $I_{\h}$, we have the following:
\begin{itemize}
\item the interior of \eqref{eq:retract}, which is $\cub(\h)^\circ$,
corresponds to $\kong{R}^n\times (0,1)^n$;
\item the boundary of \eqref{eq:retract} corresponds to the $2n$ interiors of
the codimension one faces of $\kong{R}^n\times [0,1)^n$,
each of which is $\kong{R}^n\times (0,1)^{n-1}$.
\end{itemize}
Therefore, the image of \eqref{eq:retract} under $I_{\h}$
can be retracted to the subspace of the $n$-cube $[0,1]^n$
obtaining by removing all its faces of codimension greater than one.
It is straightforward to see that (cf. Figure~\ref{fig:retraction} for $n=2$
and the shadow area in Figure~\ref{fig:LOGO})
this image can be further retracted to $\mathbf{S}_{2n}$.
By Lemma~\ref{lem:top}, we can modify the retraction
$\kong{R}^n\times [0,1)^{n-1}\to \mathbf{S}_{2n}$ near the boundaries
to get any specific contractions on the boundaries.
Thus $3^\circ$ follows.

\begin{figure}[ht]\centering
\begin{tikzpicture}[scale=2]
\draw (-1,-1) rectangle (1,1);
\draw[fill=black] (0,0) circle (.12pc);
\foreach \j in {1,...,4}
{
  \draw[fill=white] (\j*90+45:1.41) circle (.15pc);
  \draw[fill=black] (\j*90:1) circle (.11pc);
  \draw[line width=.2pc] (\j*90:1) -- (0,0);
  \draw (\j*90+45:1.2) edge[->,>=latex] (\j*90+45:.3);
  \draw (\j*90-19:.9) edge[->,>=latex,bend left=35] (\j*90-10:.66);
  \draw (\j*90-30:1) edge[->,>=latex,bend left=25] (\j*90-18:.45);
  \draw (\j*90+19:.9) edge[->,>=latex,bend right=35] (\j*90+10:.66);
  \draw (\j*90+30:1) edge[->,>=latex,bend right=25] (\j*90+18:.45);
}
\end{tikzpicture}
\caption{The retraction from a square to the star $\mathbf{S}_{4}$}
\label{fig:retraction}
\end{figure}
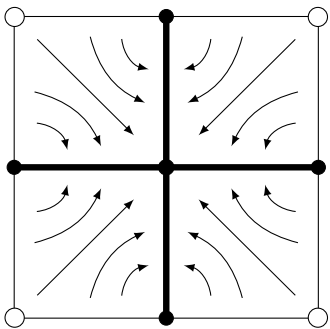

Thus, $\Stab_0-Y$ retracts to $\EG_0$,
which implies $\pi_1(\Stab_0-Y)=\pi_1(\EG_0)$.
Hence from \eqref{eq:surj} we have the surjection
$\pi_1(\EG_0)\twoheadrightarrow\pi_1(\Stab_0)$ as required.
We also see, from the retracting, that there is an embedding
$\iota:\EG_0 \hookrightarrow \Stab_0$, unique up to homotopy,
satisfying the required conditions.
\end{proof}

From now on, we will always fix a canonical embedding $\iota$ and
identify the exchange graph with its image in the space of stability conditions.

\section{Simply connectedness of Stab(Q)}\label{sec:sc.Q}
Let $Q$ be a Dynkin quiver from now on and $\Stab(Q)=\Stab(\D(Q))$.
By the connectedness of $\EG(Q)$,
we have a disjoint union $\Stab(Q)=\bigcup_{\h\in\EG(Q)}\cub(\h)$.
We aim to show the simply connectedness of $\Stab(Q)$ via $\EG(Q)$ in this section.
First, we can apply Theorem~\ref{thm:c.e.s} to $\EG(Q)$ and $\Stab(Q)$,
since clearly $\EG(Q)$ satisfies Assumption~\ref{ass} about finiteness.

\begin{corollary}\label{cor:surj}
Let $Q$ be a Dynkin quiver.
Then $\EG(Q)$ can be canonically embedded in $\Stab(Q)$ as in Theorem~\ref{thm:c.e.s}
with a surjection $\pi_1(\EG(Q))\twoheadrightarrow\pi_1(\Stab(Q))$.
\end{corollary}

Second, we prove two elementary but important lemmas.
\begin{lemma}\label{lem:lem}
Let $\h$ be a heart of $\hua{D}(Q)$ with $\Sim\h=\{S_1,...,S_n\}$ and
$\HA_{ij}=\Hom^\bullet(S_i,S_j)$.
Then for $i\neq j, j\neq k$,
\numbers
\item
    $\dim \HA_{ij}+\dim \HA_{ji} \leq 1$.
\item
    If $\HA_{ij},\HA_{jk},\HA_{ik}\neq 0$, then
    the multiplication $\HA_{ij}\otimes\HA_{jk}\to \HA_{ik}$
    is an isomorphism.
\ends
\end{lemma}

\begin{proof}
Suppose that $\HA^{\delta_1}_{ij}\neq0$ for some $\delta_1>0$.
Let $A=S_i$ and $B=S_j[\delta_1]$.
By Lemma~\ref{lem:homs}, we have
\[
    B \in \Big[\Ps(A),\Ps^{-1} \big( \tau(A[1]) \big) \Big].
\]
Thus $\HA^m_{ij}=0$ for $m\neq\delta_1$
and $\HA^m_{ji}=0$ for $m > 1-\delta_1$.
But $\HA_{ji}$ is also concentrated in positive degrees
and hence $\HA_{ji}=0$.

By Proposition~\ref{pp:standard},
there is a quiver $Q'$ such that,
$\Ps(A)$ consists of precisely the projectives in $\mod \k Q'$.
Moreover, we have $B\in\mod \k Q'$.
Let $\mathbf{b}=\dimvec B$ and $\mathbf{a}=\dimvec A$,
then we have
\begin{equation}
\label{eq:dim}
 \left\{
  \begin{array}{l}
    \dim\Hom(A,B)-\dim\Ext^1(A,B)=\<\mathbf{a},\mathbf{b}\>=\dim\HA^{\delta_1}_{ij},\\
    \dim\Hom(B,A)-\dim\Ext^1(B,A)=\<\mathbf{b},\mathbf{a}\>=\dim\HA^{\delta_1}_{ji}=0.
  \end{array}
 \right.
\end{equation}
Since $Q'$ is of Dynkin type,
the quadratic form $q(\mathbf{x})=\<\mathbf{x},\mathbf{x}\>$
is positive definite and, furthermore,
since $A\not\cong B$, we have $\mathbf{a}\neq \mathbf{b}$.
Hence
\[
    0<\<\mathbf{a}-\mathbf{b},\mathbf{a}-\mathbf{b}\>=2-\<\mathbf{a},\mathbf{b}\>
\]
i.e. $\dim\HA^{\delta_1}_{ij} \leq1$.
Thus $1^\circ$ follows.

For $2^\circ$, suppose that $\HA^{\delta_2}_{jk}\neq0$.
Since $B\in\nzero'$, Lemma~\ref{lem:homs} implies that
\[
    S_k[\delta_1+\delta_2]\in\left(\nzero'\right)[1]\cup\nzero'.
\]
Suppose that $\HA^{\delta_3}_{ik}\neq 0$ and we have
$C=S_k[\delta_3]$ is also in $\nzero'$.
Thus either $\delta_3=\delta_1+\delta_2$ or $\delta_3=\delta_1+\delta_2-1$.

Suppose that $\delta_3=\delta_1+\delta_2-1$.
Let $\mathbf{c}=\dimvec C$.
As in \eqref{eq:dim}, we have
\begin{equation*}
 \left\{
  \begin{array}{l}
       \<\mathbf{a},\mathbf{b}\>=1,\\
    \<\mathbf{b},\mathbf{a}\>=0,\\
  \end{array}
 \right.
  \left\{
  \begin{array}{l}
    \<\mathbf{a},\mathbf{c}\>=1,\\
    \<\mathbf{c},\mathbf{a}\>=0,
  \end{array}
 \right.
  \left\{
  \begin{array}{l}
    \<\mathbf{b},\mathbf{c}\>=-1,\\
    \<\mathbf{c},\mathbf{b}\>=0.
  \end{array}
 \right.
\end{equation*}
Because $A$ is simple, $\mathbf{a}\neq \mathbf{b}+\mathbf{c}$.
But $\<\mathbf{b}+\mathbf{c}-\mathbf{a},\mathbf{b}+\mathbf{c}-\mathbf{a}\>=0$,
which is a contradiction.
Therefore $\delta_3=\delta_1+\delta_2$.

Since $A$ is a simple, any non-zero $f\in \Hom(A,B)$ is injective and
so gives a short exact sequence
$0\to A \to B\to D\to 0$ in $\mod \k Q'$.
Applying $\Hom(-,C)$ to it, we get an exact sequence
\begin{gather*}
    0 \to \Hom(D,C) \to \Hom(B,C) \xrightarrow{f^*} \Hom(A,C)\to\\
     \to \Hom(D,C[1]) \to \Hom(B,C[1])=0
\end{gather*}
If $f^*$ is not an isomorphism,
then $\Hom(D,C)\neq0$ and $\Hom(D,C[1])\neq0$,
contradicting Lemma~\ref{lem:homs}.
Hence multiplication $\HA_{ij}\otimes\HA_{jk}\to \HA_{ik}$,
i.e. composition $\Hom(A,B)\otimes\Hom(B,C)\to\Hom(A,C)$,
is an isomorphism, as required.
\end{proof}

\begin{lemma}\label{lem:simple}
Let $\h$ be a heart in $\D(Q)$
and $S_i, S_j$ be two simples in $\Sim\h$.
Suppose that $\Hom^1(S_i,S_j)=0$.
Let $\h_i=\tilt{\h}{\sharp}{S_i}, \h_j=\tilt{\h}{\sharp}{S_j}$
and $\h_{ij}=\tilt{(\h_j)}{\sharp}{S_i}$.
Then
\numbers
\item
    either $\Hom^1(S_j, S_i)=0$ and we have
    $\tilt{ (\h_i) }{\sharp}{S_j}=\h_{ij}$;
\item
    or $\Hom^1(S_j, S_i)=\k$ and we have $\h_{ij}=\tilt{  (\h_*)  }{\sharp}{S_j}$,
    where $T_j=\phi^{-1}_{S_i}(S_j)$ and $\h_*=\tilt{ (\h_i) }{\sharp}{T_j}$.
\ends
\end{lemma}
\begin{gather}\label{eq:45}{
\xymatrix@C=1.4pc@R=1pc{
    &   \h_i    \ar[dr]^{S_j}\\
    \h  \ar[ur]^{S_i} \ar[dr]_{S_j} &&  \h_{ij}\\
    &   \h_j    \ar[ur]_{S_i}
}\qquad
\xymatrix@C=1pc@R=0.8pc{
    &   \h_i    \ar[rr]^{T_j} &&      \h_* \ar[dr]^{S_j} \\
    \h  \ar[ur]^{S_i}\ar[ddrr]_{S_j}  &&&&   \h_{ij}\\\\
    &&   \h_j   \ar[uurr]_{S_i}
}}
\end{gather}
\begin{proof}
By Lemma~\ref{lem:lem},
we have $\dim\Hom^\bullet(S_j,S_i)\leq1$.
Thus $\Hom^1(S_j, S_i)$ equals zero or $\k$.
We only prove the second case while the first one is much simpler.
By \cite[Proposition~5.2]{Q},
we know how the simples change during tilting.
Since the simples determine a heart, we only need to show that
for any other simple $X$ in $\h$, it turns to the same simple
in $\h_{ij}$ and $\tilt{  (\h_*)  }{\sharp}{S_j}$.
By Lemma~\ref{lem:lem}, we have the following cases.
\begin{itemize}
\item
$\Hom^\bullet(S_i,X)=\Hom^\bullet(X,S_i)=0$
or $\Hom^\bullet(S_j,X)=\Hom^\bullet(X,S_j)=0$.
\item $\Hom^\bullet(X,S_i)=\Hom^\bullet(X,S_j)=0$ and
\[
    \Hom^\bullet(S_i,X)=\k[-t],\quad\Hom^\bullet(S_j,X)=\k[-t-1]
\]
for some $t>1$.
Moreover, we have an isomorphism
\[
    \Hom^{t+1}(S_j,X)\cong\Hom^{1}(S_j,S_i)\otimes\Hom^{t}(S_i,X).
\]
\item $\Hom^\bullet(S_i,X)=\Hom^\bullet(S_j,X)=0$ and
\[
    \Hom^\bullet(X,S_j)=\k[-t],\quad \Hom^\bullet(X,S_i)=\k[-t-1]
\]
for some $t>1$.
Moreover, we have an isomorphism
\[
    \Hom^{t+1}(X,S_j)\cong\Hom^{t}(X,S_i)\otimes\Hom^{1}(S_i,S_j).
\]
\end{itemize}
By the formula in \cite[Proposition~5.2]{Q},
a direct calculation shows that $X$ indeed becomes the same simple
in $\h_{ij}$ and $\tilt{  (\h_*)  }{\sharp}{S_j}$, in any of the cases above,
which completes the proof.
\end{proof}

\begin{remark}
An alternative proof, which covers more general situation,
can be found in \cite{Qiu-Woolf}.
\end{remark}

Next, we discuss the fundamental group of $\EG(Q)$.

\begin{proposition}\label{pp:45}
If $Q$ is of Dynkin type,
then $\pi_1(\EG_N(Q,\nzero))$ is generated by
squares and pentagons as in \eqref{eq:45} for any $N\geq 2$.
Further, $\pi_1(\EG(Q))$ is generated by such squares and pentagons.
\end{proposition}
\begin{proof}
For any cycle $c$ in $\EGp_N(Q,\nzero)$,
\[
    D(c)=\{\h \mid \exists \h'\in c, \h'\leq \h\leq \nzero[N-2]\}
\]
is finite, by Proposition~\ref{pp:lim}.
We use induction on $\#D(c)$ to prove that any cycle $c$ in $\EGp_N(Q,\nzero)$
is generated by squares and pentagons and hence the first statement will follow.
If $\#D(c)=1$, then $c$ is trivial.
Suppose that $\#D(c)>1$ and
any cycle $c'\subset\EGp_N(Q,\nzero)$ with $\#D(c')<\#D(c)$
is generated by the squares and pentagons.
Choose a source $\h$ in $c$ such that $\h'\nless\h$ for any other source $\h'$ in $c$.
Let $S_i$ and $S_j$ be the arrows coming out at $\h$.
If $i=j$ we can delete them in $c$ to get a new cycle $c'$.
If $i\neq j$,
we know that $S_i:\h\to\h_i$ and $S_j:\h\to\h_j$ are either in a square or a pentagon
as in \eqref{eq:45}.
By the second part of \cite[Lemma~5.4]{Q},
we know that $\Ho{N-1}(S_i)=0$ and hence
$\h_{ij}=\tilt{(\h_j)}{\sharp}{S_i}\in\EGp_N(Q,\nzero)$.
Thus this square or pentagon is in $\EGp_N(Q,\nzero)$
and we can replace $S_i$ and $S_j$ in $c$ by other edges in this square or pentagon
to get a new cycle $c'\subset\EGp_N(Q,\nzero)$.
Either way, we have $D(c')\subset ( D(c)-\{\h\} )$ for the new cycle $c'$
and we are done.

Now choose any cycle $c$ in $\EG(Q)$.
By \eqref{eq:lim eg},
all hearts in $c[k]$ are in $\EGp_N(Q,\nzero)$ for some integer $k$ and $N\gg1$.
Then the second statement follows from the first one.
\end{proof}

We precede to show that the generators for $\EG(Q)$ are trivial in $\Stab(Q)$.

\begin{lemma}\label{lem:sp}
Any square or pentagon as in \eqref{eq:45} is trivial in $\pi_1(\Stab(Q))$.
\end{lemma}
\begin{proof}

\begin{figure}[hb]\centering
\begin{tikzpicture}[scale=1]
\draw[dotted,->,>=latex] (-2,0) -- (6,0) node[right]{$x$};
\draw[dotted,->,>=latex] (0,-1) -- (0,4) node[above]{$y$};
\draw (0,0) node[below left]{$0$};
\draw[thick,->,>=latex] (0,0) -- (180/2:3)    node[left]{$Z(S_k)$};
\draw[dotted,thick] (0,0) -- (4*180/17:5.5)  node[above right] {$$};
\draw[dotted,thick] (0,0) -- (180+4*180/17:1.5)  ;
\draw[thick,->,>=latex] (0,0) -- (180/17:3)   node[right]{$Z(S_i)$};
\draw[thick,->,>=latex] (0,0) -- (3*180/17:3) node[above right]{$Z(S_j)$};
\draw[thick,->,>=latex] (0,0) -- (2*180/17:5.9) node[right]{$Z(T_j)$};
\end{tikzpicture}
\caption{}\label{fig:sc}
\end{figure}

Recall that $\EG(Q)$ can be embedded into $\Stab(Q)$ up to homotopy,
by Corollary~\ref{cor:surj}.
We claim that, up to homotopy,
the image of a pentagon or a square (starting as a heart $\h$)
is contained in the contractible prism
\[
    \mathbf{P}=\kong{C}\cdot\cub(\h)\cong \kong{C}\cdot H^n,
\]
where $H$ is the upper half plane in \eqref{eq:H}.
If so, the lemma follows.

For the pentagon case, suppose that we are in the situation
of case $2^\circ$ of Lemma~\ref{lem:simple}.
Let $\Sim\h=\{S_1,...,S_n\}$.
Consider the stability condition $\sigma$ with heart is $\h$ determined by
\[\begin{cases}
    Z(S_k)=\exp(\frac{1}{2}\pi\mathbf{i}) \quad  k\neq i,j,\\
    Z(S_i)=\exp(\delta\pi\mathbf{i}),\\
    Z(S_j)=\exp(3\delta\pi\mathbf{i}),
\end{cases}\]
for some $\delta>0$.
Since $\dim\Hom^1(S_j,S_i)=1$,
there is an unique extension $T_j$ of $S_j$ on top of $S_i$.
Moreover, $T_j$ has phase $2\delta$.
Thus we can choose $\delta$ so small that
any stable object other than $S_i, T_j$ and $S_j$ has phase larger than $4\delta$
(cf. Figure~\ref{fig:sc}).
Consider the interval
$L=\{\sigma_\varepsilon\}_{\varepsilon\in[-4\delta,0]}$,
where $\sigma_\varepsilon=\varepsilon\cdot\sigma$.
We have
\[\begin{cases}
       \sigma_\varepsilon\in\cub(\h),   & \varepsilon\in(-\delta,0],\\
       \sigma_\varepsilon\in\cub(\h_i),   & \varepsilon\in(-2\delta,-\delta),\\
       \sigma_\varepsilon\in\cub(\h_*),   & \varepsilon\in(-3\delta,-2\delta),\\
       \sigma_\varepsilon\in\cub(\h_{ij}),   & \varepsilon\in[-4\delta,-3\delta).
\end{cases}\]
Therefore $L$ is homotopic to the path $\h\to\h_i\to\h_*\to\h_{ij}$ in $\EG(Q)$
by Corollary~\ref{cor:surj}.
Similarly, consider the stability condition $\sigma'$ with heart is $\h$ determined by
\[\begin{cases}
    Z'(S_k)=\exp(\frac{1}{2}\pi\mathbf{i}) \quad  k\neq i,j,\\
    Z'(S_i)=\exp(3\delta'\pi\mathbf{i}),\\
    Z'(S_j)=\exp(\delta'\pi\mathbf{i}),
\end{cases}\]
for some $\delta'>0$.
Then when $\delta'$ is very small, the interval
$L'=\{\sigma'_\varepsilon\}_{\varepsilon\in[-4\delta',0]}$ is homotopic to
the path $\h\to\h_j\to\h_{ij}$ in $\EG(Q)$ by Corollary~\ref{cor:surj}.

Notice that the end points of $L$ are
in (the interior of) $\cub(\h)\cap\mathbf{P}$ and $\cub(\h_{ij})\cap\mathbf{P}$ respectively.
So are the end points of $L'$.
Thus, we can choose two paths in $\cub(\h)\cap\mathbf{P}$ and $\cub(\h_{ij})\cap\mathbf{P}$
respectively, connecting $L$ and $L'$ to get a circle $c$.
Then $c$ is homotopic the pentagon in \eqref{eq:45}
and contained in $\mathbf{P}$, as required.

Similarly for the square case.
\end{proof}

We end this section by proving the simply connectedness of $\Stab(Q)$.

\begin{theorem}\label{thm:sc1}
If $Q$ is of Dynkin type, then $\Stab(Q)$ is simply connected.
\end{theorem}
\begin{proof}
By Proposition~\ref{pp:45} and Lemma~\ref{lem:sp},
we know that $\pi_1(\EG(Q))$ is trivial in $\Stab(Q)$.
Then the theorem follows from the surjection in Corollary~\ref{cor:surj}.
\end{proof}

\section{Simply connectedness of Calabi-Yau case}\label{sec:sc.CY.Q}
\subsection{The principal component}\label{sec:stab pc}
In this subsection,
we show that $\EGp(\qq{N})$ induces
a connected component in the space of stability conditions $\Stab(\D(\qq{N}))$.

\begin{lemma}\label{lem:cond}
$\EGp_3(\qq{N},\h)$ is finite, for any heart $\h\in\EGp(\qq{N})$.
\end{lemma}
\begin{proof}
If $N=2$, all hearts in $\EGp(\qq{2})$ are equivalent since
any simple tilting is the same to apply a spherical twist.
Moreover, any such a heart has only finitely many torsion pairs
due to \cite[Theorem~4.1]{Mu}, which implies the lemma.

Next, we assume that $N\geq3$.
By \eqref{eq:eg}, we can assume that $\h\in\EGp_N(\qq{N},\zero)$
without lose of generality.
By Theorem~\ref{thm:Q}, we have isomorphism \eqref{eq:I1}
and hence $\EGp_N(\qq{N},\zero)$ is finite by Proposition~\ref{pp:lim}.

Now we claim that,
for $\h\in\EGp_3(\qq{N},\zero)$,
if $\EGp_3(\qq{N},\h_0)$ is finite for any $\nzero\leq\h_0<\h$,
then $\EGp_3(\qq{N},\h)$ is also finite.

If $\h\in\EGp_{N-1}(\qq{N},\zero)$, then
$\EGp_{3}(\qq{N},\h)\subset\EGp_{N}(\qq{N},\zero)$,
which implies that $\EGp_{3}(\qq{N},\h)$ is finite.
Now suppose that $\h\notin\EGp_{N-1}(\qq{N},\zero)$.
Let $\h$ is induced from $\nh\in\EGp_N(Q,\h_Q)$ via $\hua{I}$,
and we have $\nh\notin\EGp_{N-1}(Q,\nzero)$ by \eqref{eq:I1}.
By \eqref{eq:ar}, for any simple $\widehat{S}\in\Sim\nh$,
there is some integer $m$ such that $\widehat{S}\in\nzero[m]$;
and we have $0\leq m\leq N-2$ by \cite[Lemma~5.4]{Q}.
Since $\nh\notin\EGp_{N-2}(Q,\nzero)$,
there exists a simple $\widehat{S}\in\Sim\h$ such that
$\Ho{N-2}(\widehat{S})\neq0$, where $\Ho{\bullet}$ is with respect to $\nzero$.
By \eqref{eq:ar}, $\widehat{S}\in\nzero[N-2]$.
Then $S=\hua{I}(\widehat{S})\in\zero[N-2]$.
By \cite[Lemma~5.4]{Q}, we have
\[
    l(\h,S)\cap\EGp_N(\qq{N}, \zero)=\{\tilt{\h}{i\flat}{S}\}_{i=0}^{N-2}.
\]
By the inductive assumption, we know that
$\EGp_3(\qq{N},\tilt{\h}{\flat}{S})$ and $\EGp_3(\qq{N},\tilt{\h}{(N-2)\flat}{S})$ is finite.
Thus, so is
\[
    \EGp_3(\qq{N},\tilt{\h}{\sharp}{S})
        =\phi^{-1}_S\EGp_3(\qq{N},\tilt{\h}{(N-2)\flat}{S}),
\]
where we use the fact that $\tilt{\h}{\sharp}{S}=\phi^{-1}(\tilt{\h}{(N-2)\flat}{S})$ by \cite[(8.3)]{Q}.
By \cite[Proposition~9.1]{Q}, we have
\[
    \EGp_3(\qq{N},\h)\subset\left( \EGp_3(\qq{N},\tilt{\h}{\flat}{S})
        \cup \EGp_3(\qq{N},\tilt{\h}{\sharp}{S}) \right),
\]
which implies the finiteness of $\EGp_3(\qq{N},\h)$.
Therefore the lemma follows by induction.
\end{proof}

\begin{proposition}\label{pp:EGp=EG}
$\EGp_3(\qq{N},\h)=\EG_3(\qq{N},\h)$, for any heart $\h\in\EGp(\qq{N})$.
\end{proposition}
\begin{proof}
Suppose that there exists a heart $\h'\in\EG_3(\qq{N},\h)-\EGp_3(\qq{N},\h)$,
we claim that there is an infinite directed path
\[
    \h_1 \xrightarrow{S_1} \h_2 \xrightarrow{S_2} \h_3 \to \cdots
\]
in $\EGp_3(\qq{N},\h)$ satisfying $\h_j<\h'$ for any $j\in\kong{N}$.

Use induction starting from $\h_1=\h$.
Suppose we have $\h_j\in\EGp_3(\qq{N},\h)$ such that $\h_j<\h'$.
If for any simple $S\in\h_j$, we have $S\in\h'$,
then $\h'\supset\h_j$ which implies $\hua{P}'\supset\hua{P}_j$, or $\h'\leq\h_j$;
this contradicts to $\h_j<\h'$.
Thus there is a simple $S_j\in\h_j$ such that $S_j\notin\h'$.
Notice that $\h_j<\h'\leq\h[1]\leq\h_j[1]$,
then by \cite[Proposition~9.1]{Q},
we have $\h_{j+1}=\tilt{(\h_j)}{\sharp}{S_j}\leq\h'(\leq \h[1])$.
Notice that $\h'\notin\EGp_3(\qq{N},\h)$, therefore $\h_{j+1}\neq\h'$, which implies the claim.

Then $\EGp_3(\qq{N},\h)$ is infinite,
which contradicts to the finiteness in Lemma~\ref{lem:cond}.
\end{proof}

Now, we can identify the principal component of $\Stab(\D(\qq{N}))$ as follows,
which is the connected component containing $\cub(\zero)$.

\begin{corollary}\label{cor:CYstab}
Let $Q$ be a Dynkin quiver and $N\geq2$.
Then there is a principal component
\[
    \Stap(\qq{N})=\bigcup_{\h\in\EGp(\qq{N})}\cub(\h)
\]
in $\Stab(\D(\qq{N}))$.
Moreover, $\EGp(\qq{N})$ can be canonical embedded in $\Stap(\qq{N})$
as in Theorem~\ref{thm:c.e.s}
with a surjection $\pi_1(\EGp(\qq{N}))\twoheadrightarrow\pi_1(\Stap(\qq{N}))$.
\end{corollary}
\begin{proof}
By Lemma~\ref{lem:well known},
there is a bijection between the set of torsion pairs in $\h$
and the set $\EG_3(\qq{N},\h)$.
Thus, Lemma~\ref{lem:cond} and Proposition~\ref{pp:EGp=EG} imply that
any heart in $\EGp(\qq{N})$ has only finitely many torsion pairs.
Moreover, \cite[Corollary~8.4]{Q} says that any heart in $\EGp(\qq{N})$ is finite.
Thus, Assumption~\ref{ass} holds for $\EGp(\qq{N})$;
and Theorem~\ref{thm:woolf} and Theorem~\ref{thm:c.e.s} gives the theorem.
\end{proof}

Note that the gluing structure of $\Stap(\qq{N})$ is also encoded by
the formula \eqref{eq:boundary}.

\subsection{Simply connectedness}
Define the \emph{basic cycles} in $\Stap(\qq{N})/\Br$ to be
braid group orbits of lines (cf. Definition~\ref{def:line}) in $\Stap(\qq{N})$.

\begin{theorem}\label{thm:sc2}
Suppose that $Q$ is of Dynkin type and let $\h\in\EGp(\qq{N})$.
Then
\[\pi_1(\Stap(\qq{N})/\Br,[\h])\]
is generated by
basic cycles containing the class
$[\h]$ and it is a quotient group of the braid group $\Br_Q$ (cf. \cite[Definition~2.2]{Br-T}).
\end{theorem}
\begin{proof}
If $N=2$, $\EGp(\qq{2})/\Br$ consists of a single orbit and the theorem follows from
Theorem~\ref{thm:c.e.s} directly.
Now assume $N\geq3$.

Let $\Sim\h=\{S_1,...S_n\}$, $\phi_{k}=\phi_{S_k}$ and
let $c_k$ be the basic cycle corresponding to $l(\h, S_k)$, for $k=1,...,n$.
Denote by $p$ the quotient map
\[p:\Stap(\qq{N})\to\Stap(\qq{N})/\Br.\]
We will drop $X$ in the notation $\pi_1(X, x)$ if there is no ambiguity.
By \cite[Theorem~13.11]{F}, we have a short exact sequence
\begin{gather}\label{eq:BRses}
    \xymatrix{
    0 \ar[r]& p_*\left(  \pi_1(\h)  \right)
    \ar[r]& \pi_1([\h]) \ar[r]^{\varrho} & \Br(\qq{N}) \ar[r]& 0,
}\end{gather}
where $\varrho$ sends $c_k$ to $\phi^{-1}_{k}$.
To prove the theorem, it is sufficient to show that
$\{c_k\}$ satisfies the braid group relation and generates $\pi_1([\h])$.

First, let $i$ and $j$ be a pair of nonadjacent vertices in $Q$.
Then $\Hom^\bullet(S_i,S_j)=\Hom^\bullet(S_j,S_i)=0$ and
we need to show that $c_i  c_j =c_j  c_i$ in $\pi_1([\h])$.
To do so, consider the lifting $L_1$ of $c_i c_j  c_i^{-1} c_j^{-1}$
in $\pi_1(\h)$ starting at $\h$. Let
\begin{gather*}
    \h^i=\phi^{-1}_i (\h), \quad\h^{ji}=\phi^{-1}_j\circ \phi^{-1}_i (\h),\\
    \h^j=\phi^{-1}_j (\h), \quad\h^{ij}=\phi^{-1}_i\circ \phi^{-1}_j (\h)
\end{gather*}
and we have $\h^{ij}=\h^{ji}$ in this case.
Then $L_1\in\pi_1(\h)$ is the boundary in Figure~\ref{3*3} with clockwise orientation.
By Lemma~\ref{lem:lem} and \cite[Theorem~8.1]{Q}
We know that $\Hom^\bullet(S_j,S_i)$ is concentrated in one degree with dimension at most one.
By the iterated application of \cite[Proposition~5.2]{Q},
$L_1$ is the sum of $(N-1)^2$ squares, each of which is as in \eqref{eq:45}.
For instance, Figure~\ref{3*3} is the CY-$4$ case,
where the blue (resp. red) edges have direction-$S_i$ (resp. direction-$S_j$)
and the hearts are uniquely determined by these edges.
Using the same argument as in Lemma~\ref{lem:sp},
we see that any such square is trivial in $\pi_1(\h)$.
Thus $L_1$ is trivial in $\pi_1(\h)$,
or equivalently, $c_i  c_j =c_j  c_i$ in $\pi_1([\h])$ as required.

\begin{figure}[ht]\centering
\begin{tikzpicture}[scale=1.4]
\draw[fill=gray!14]
    (1,1) rectangle (4,2);
\draw[white, thick]
    (1,1) rectangle (4,2);
\foreach \k in {1,...,4}{
  \foreach \j in {1,...,3}{
    \path (\k,\j+1) node (t) {$\circ$};
    \path (\k,\j) node {$\circ$} edge[->,thick,>=latex,NavyBlue] (t);
    \path (\j+1,\k) node (t) {$\circ$};
    \path (\j,\k) node {$\circ$} edge[->,thick,>=latex,red] (t);
  }
}
\path (1,1) node[below left] {$\h$};\path (1,1) node {$\bullet$};
\path (1,4) node[above left] {$\h^i$};\path (1,4) node {$\bullet$};
\path (4,1) node[below right] {$\h^j$};\path (4,1) node {$\bullet$};
\path (4,4) node[above right] {$\h^{ij}$};\path (4,4) node {$\bullet$};
\end{tikzpicture}
\caption{Square cover of $L_1$, CY-$4$ case}\label{3*3}
\end{figure}
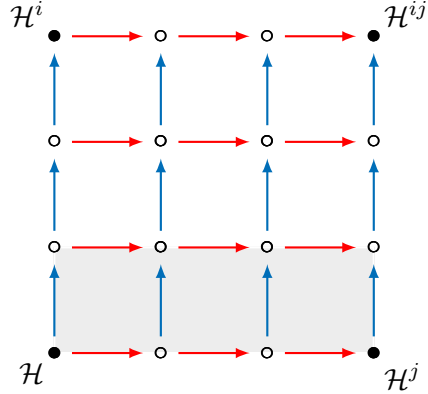

Second, let $i$ and $j$ be a pair of adjacent vertices in $Q$.
Without loss of generality, let the arrow be $j\to i$.
Then $\Hom^\bullet(S_j,S_i)=\k[-1]$ and
we need to show that $c_i  c_j c_i =c_j c_i c_j$ in $\pi_1([\h])$.
To do so, consider the lifting $L_2$ of $c_i  c_j c_i c_j^{-1} c_i^{-1} c_j^{-1}$
in $\pi_1(\h)$ starting at $\h$.
Let $\h^i,\h^j$ as before and
\begin{gather*}
    T=\phi^{-1}_i(S_j),\quad R=\phi^{-1}_j(S_i),\quad
    \h'=\phi^{-1}_j \circ \phi^{-1}_T \circ \phi^{-1}_i(\h).
\end{gather*}
By \cite[Lemma~2.11]{ST}, we have
\[
    \phi^{-1}_j \circ \phi^{-1}_T \circ \phi^{-1}_i
    =\phi^{-1}_i \circ \phi^{-1}_R \circ \phi^{-1}_j.
\]
Then $L_2\in\pi_1(\h)$ is the boundary in Figure~\ref{3*5} with clockwise orientation.
Similarly,
$L_2$ is the sum of $(N-1)(2N-3)$ pentagons/squares,
each of which is as in \eqref{eq:45}.
For instance, Figure~\ref{3*5} is the CY-$4$ case,
where the blue (resp. red, dashed and dotted) edges have
direction-$S_i$ (resp. direction-$S_j$, direction-$T$ and direction-$R$).
Thus $L_2$ is trivial in $\pi_1(\h)$ as above,
or equivalently, $c_i  c_j c_i =c_j c_i c_j$ in $\pi_1([\h])$ as required.

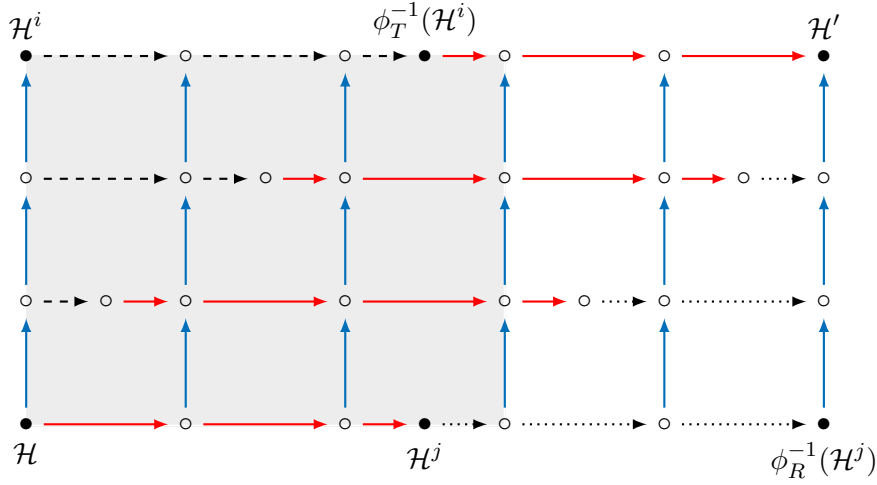
\begin{figure}[ht]\centering
\begin{tikzpicture}[scale=1.5]
\draw[fill=gray!14]
    (0,.86) rectangle (4.2cm,4.14cm);
\draw[white, thick]
    (0,.86) rectangle (4.2cm,4.14cm);
\matrix (m) [matrix of math nodes, row sep=1.2cm, column sep=.6cm] at (3.5,2.5)
    {\bullet&&\circ&&\circ& \bullet &\circ&&\circ&&\bullet\\
    \circ&&\circ& \circ &\circ&&\circ&&\circ& \circ &\circ\\
    \circ& \circ &\circ&&\circ&&\circ& \circ &\circ&&\circ\\
    \bullet&&\circ&&\circ& \bullet &\circ&&\circ&&\bullet\\};
\path (m-1-1) \es (m-1-3);\path (m-1-3) \es (m-1-5);\path (m-1-5) \es (m-1-6);
\path (m-1-6) \ee (m-1-7);\path (m-1-7) \ee (m-1-9);\path (m-1-9) \ee (m-1-11);
\path (m-2-1) \es (m-2-3);\path (m-2-3) \es (m-2-4);\path (m-2-4) \ee (m-2-5);
\path (m-2-5) \ee (m-2-7);\path (m-2-7) \ee (m-2-9);\path (m-2-9) \ee (m-2-10);
\path (m-2-10) \eo (m-2-11);\path (m-3-1) \es (m-3-2);
\path (m-3-2) \ee (m-3-3);\path (m-3-3) \ee (m-3-5);\path (m-3-5) \ee (m-3-7);
\path (m-3-7) \ee (m-3-8);\path (m-3-8) \eo (m-3-9);\path (m-3-9) \eo (m-3-11);
\path (m-4-1) \ee (m-4-3);\path (m-4-3) \ee (m-4-5);\path (m-4-5) \ee (m-4-6);
\path (m-4-6) \eo (m-4-7);\path (m-4-7) \eo (m-4-9);\path (m-4-9) \eo (m-4-11);
\foreach \k in {1,3,5,7,9,11}{
    \path (m-4-\k) \eb (m-3-\k);\path (m-3-\k) \eb (m-2-\k);\path (m-2-\k) \eb (m-1-\k);
}
\path (0,1-.2) node[below] {$\h$};
\path (3.5,1-.2) node[below] {$\h^j$};
\path (7,1-.2) node[below] {$\phi^{-1}_R(\h^j)$};
\path (0,4.2) node[above] {$\h^i$};
\path (3.5,4.2) node[above] {$\phi^{-1}_T(\h^i)$};
\path (7,4.2) node[above] {$\h'$};
\end{tikzpicture}
\caption{Square and pentagon cover of $L_2$, CY-$4$ case}\label{3*5}
\end{figure}

Therefore, we have shown that $\{c_k\}$ satisfies the braid group relation.
To finish, we only need to show that $\{c_k\}$ generates $\pi_1([\h])$.
By Theorem~\ref{thm:Q},
we have $\EGp(\qq{N})/\Br\cong\cc{\EGp_N}(\qq{N},\zero)$
and hence $\pi_1(\EGp(\qq{N})/\Br)$ is generated by
all squares and pentagons in $\EGp_N(\qq{N},\zero)$ and all basic cycles (cf. \cite[\S~5.3]{Q}).
Using again the argument of Lemma~\ref{lem:simple},
these squares and pentagons are trivial as in Lemma~\ref{lem:sp}.
Then by induction, what is left to show is that
any other basic cycle that does not contain $[\h]$ is generated by $\{c_k\}$.
Consider the basic cycle $s_i c_T s_i^{-1}$ for demonstration,
where $s_i$ is the path from $\h$ to $\tilt{\h}{\sharp}{S_i}$, $T=\phi^{-1}_i(S_j)$
and $c_T$ be the basic cycle induced by the line $l(\tilt{\h}{\sharp}{S_i}, T)$.

In the ($i$ and $j$) nonadjacent case,
let $L_3$ be the lifting of $(s_i c_T   s_i^{-1}) c_i^{-1}$
in $\pi_1(\h)$ starting at $\h$.
Then $L_3$ is the sum of the gray squares in Figure~\ref{3*3}
(which is a partial sum of $L_1$).
As above, $L_3$ is trivial in $\pi_1(\h)$, or equivalently, $s_i c_T s_i^{-1}= c_i$.
In the adjacent case,
let $L_4$ be the lifting of
$c_j (s_i c_T   s_i^{-1})   c_j^{-1}  c_i^{-1} \in \pi_1([\h])$
in $\pi_1(\h)$ starting at $\h$.
Then $L_4$ is the sum of the gray squares in Figure~\ref{3*5}
(which is a partial sum of $L_2$).
As above, $L_4$ is trivial in $\pi_1(\h)$,
or equivalently, $s_i c_T   s_i^{-1}= c_j^{-1} c_i  c_j$.
Either way, $s_i c_T s_i^{-1}$ is generated by $\{c_k\}_{k=1}^n$, as required.
\end{proof}

\begin{corollary}\label{cor:conn}
Let $Q$ be a Dynkin quiver.
If the braid group action on $\D(\qq{N})$ is faithful,
i.e. $\Br(\qq{N})\cong\Br_Q$,
then $\Stap(\qq{N})$ is simply connected.
In particular, this is true for $Q$ of type $A_n$ or $N=2$.
\end{corollary}
\begin{proof}
If $\Br(\qq{N})\cong\Br_Q$,
then $\varrho$ in \eqref{eq:BRses} is an isomorphism
by the second part of Theorem~\ref{thm:sc2}.
Hence $\pi_1(\Stap(\qq{N}))=1$, which implies the simply connectedness.
The faithfulness for $Q$ of type $A_n$ is shown in \cite{ST}
and faithfulness for $N=2$ is shown in \cite{Br-T}.
\end{proof}

\begin{remark}\label{rem:stab}
By Theorem~\ref{thm:sc2},
basic cycles in $\Stap(\qq{N})/\Br$ are the generators of its fundamental group,
which provide a topological realization of almost completed cluster tilting objects
(cf. \cite[Remark~8.9]{Q}).
In fact, our philosophy is that $\Stap(\qq{N})/\Br$
is the `complexification' of the dual of cluster complex
and provides the `right' space of stability conditions for the higher cluster category
\[
    \C{N-1}{Q}=\hua{D}(Q)/\shift{N-1},
\]
where $\shift{N-1}=\tau^{-1}\circ[N-2]\in\Aut\D(Q)$.
Notice that there are no hearts in $\C{N-1}{Q}$ and
thus the space of stability conditions $\Stab(\C{N-1}{Q})$ is empty in the usual sense.
\end{remark}

Here are two sensible conjectures.

\begin{conjecture}
For any acyclic quiver $Q$, $\Br(\qq{N})\cong\Br_Q$.
\end{conjecture}

\begin{conjecture}
For a Dynkin quiver $Q$,
$\Stab(\D(Q))$ and $\Stap(\D(\qq{N}))$ are contractible.
\end{conjecture}

\section{A limit formula}\label{sec:sc limit}
In this section, we provide a limit formula for spaces of stability conditions.

\begin{lemma}\label{lem:induce}
If $\h=\hua{L}_*(\nh)$ for some heart $\nh\in\EG(Q)$,
then a stability condition $\widehat{\sigma}=(\widehat{Z},\widehat{\hua{P}})$ on $\D(Q)$
with heart $\nh$ canonically induces a stability condition
$\sigma=(Z,\hua{P})$ with heart $\h$
and such that $Z(\hua{L}(\widehat{S}))=\widehat{Z}(\widehat{S})$
for any $\widehat{S}\in\Sim\nh$.
Thus we have a homeomorphism $\hua{L}_*:\cub(\nh)\to\cub(\h)$.
\end{lemma}
\begin{proof}
The heart $\nh$ and $\h$ are both finite by \cite[Theorem~5.7 and Corollary~8.4]{Q}.
Then we have the following diagram
\begin{gather}\label{eq:cdd}
\xymatrix{
    \cub(\nh)  \ar@{.}[r]\ar[d] & \cub(\h)  \ar[d] \\
    \Hom(\K(\D(Q)), \kong{C}) \ar@{-}[r]^{\cong} &  \Hom(\K(\D(\qq{N})), \kong{C}),
}\end{gather}
where the vertical homeomorphisms are given by the second part of Theorem~\ref{thm:B1}
and the horizontal one is induced by the isomorphism between the
corresponding Grothendieck groups.
Therefore we have the required the homeomorphism $\hua{L}_*$ by composing the maps
in \eqref{eq:cdd}.
\end{proof}

\begin{theorem}
\label{thm:limit}
We have
\[\Stab(Q) \cong \lim_{N\to\infty} \Stap(\qq{N})/\Br(\qq{N})\]
in the following sense:
\numbers
\item
    There exists a family of open subspaces $\{\hua{S}_N\}_{N\geq2}$ in $\Stap(Q)$
    satisfying $\hua{S}_N\subset\hua{S}_{N+1}$ and
    $\Stab(Q) \cong \lim_{N\to\infty} \hua{S}_N$.
\item
    $\hua{S}_N$ is homeomorphic to a fundamental domain for $\Stap(\qq{N})/\Br$.
\ends
\end{theorem}
\begin{proof}
Let $\Stap_N(Q)$ and $\Stap_N(\qq{N})$ be the interior of
\begin{gather*}
    \bigcup_{\h\in\EGp_N(Q,\nzero)}\cc{\cub(\h)}
    \quad\text{and}\quad
    \bigcup_{\h\in\EGp_N(\qq{N},\zero)}\cc{\cub(\h)}
\end{gather*}
respectively.
By \eqref{eq:boundary},
we know that a face $F_Q$ of some cell $\cub(\nh)$ is in $\Stap_N(Q)$
if and only if \[F_Q=\cc{\cub(\nh)}\cap\cub(\nh')\]
for some $\nh,\nh'\in\EGp_N(Q,\nzero)$ satisfying $\nh[-1]\leq\nh'<\nh$.
Similarly, a face $F_\Gamma$ of some cell $\cub(\h)$ is in $\Stap_N(\qq{N})$
if and only if \[F_\Gamma=\cc{\cub(\h)}\cap\cub(\h')\]
for some $\h,\h'\in\EGp_N(\qq{N},\zero)$ satisfying $\h[-1]\leq\h'<\h$.
By Lemma~\ref{lem:induce},
we know that any such face $F_\Gamma$ in $\Stap_N(\qq{N})$ is induced
from some face $F_Q$ in $\Stap_N(Q)$ via the L-immersion $\hua{I}$ as in Theorem~\ref{thm:Q},
in the sense that we have
\[
    \hua{I}_*(F_Q)=\hua{I}_*\left(\cc{\cub(\nh)}\cap\cub(\nh')\right)
    =\cc{\hua{I_*}(\cub(\nh))}\cap\hua{I}_*(\cub(\nh'))
    =\cc{\cub(\h)}\cap\cub(\h')
    =F_\Gamma.
\]
Thus we can glue the homeomorphisms in Lemma~\ref{lem:induce} to a homeomorphism
\[\hua{I}_*:\Stap_N(Q)\to\Stap_N(\qq{N}).\]

Let $\hua{S}_N=m\cdot\Stap_N(Q)$, for $m=\lfloor -\frac{N}{2} \rfloor$,
where $\cdot$ is the natural $\kong{C}$-action on the space of stability conditions
defined in \S~\ref{sec:SC}.
Then $1^\circ$ follows from the limit formula in Proposition~\ref{pp:lim}
and we have homeomorphisms
\[\hua{S}_N\cong\Stap_N(Q)\cong\Stap_N(\qq{N}),\]
which completes the proof.
\end{proof}

\begin{example}
The calculations of $\Stab(A_2)$ and $\Stap(\Gamma_N A_2)$
in \cite{BQS} (cf. \cite[p16, Figure~2]{ACV})
illustrate the idea of the limit in Theorem~\ref{thm:limit} in the $A_2$ case.
\end{example}

\section{Directed paths and HN-strata}\label{sec:DPHN}
In this section, we will study the relations between
directed paths in the exchange graph $\EG(Q)$,
HN-strata for $\nzero$,
slicings on $\D(Q)$ and stability functions on $\nzero$.

\subsection{Directed paths}\label{sec:DP}
Let $\EG(Q;\h_1,\h_2)$ be the full subgraph of $\EG(Q)$
consisting of hearts $\h_1\leq\h\leq\h_2$.
Denote by $\dpath(\h_1,\h_2)$
the set of all directed paths from $\h_1$ to $\h_2$ in $\EG(Q;\h_1,\h_2)$.

\begin{lemma}
Suppose $\h_1\leq\h_2$.
Then $\dpath(\h_1,\h_2)\neq\emptyset$
if at least one of $\h_1$ and $\h_2$ is standard.
In particular, we have
\[
    \EG(Q;\h,\h[N-2])=\EG_N(Q,\h)=\EGp_N(Q,\h),
\]
for any standard heart $\h\in\EG(Q)$.
\end{lemma}
\begin{proof}
Without loss of generality, suppose that $\h_1=\nzero[1]$ which is standard.
For any simple $S_i\in\Sim\h_2$,
$S_i\in\nzero[m_i]$ for some integer $m_i$ by \eqref{eq:ar}.
Since $\h_1\leq\h_2$, we have $m_i\geq1$.
Choose $N\gg1$ such that $\h_2\in\EGp_N(Q,\nzero)$ and then
$\#\Ind(\hua{P}_1-\hua{P}_2)$ is finite.
If $\h_1<\h_2$, there exists $j$ such that $m_j>1$.
By \cite[Lemma~5.4]{Q}, we can backward tilt $\h_2$ to $\tilt{(\h_2)}{\flat}{S_j}$
within $\EGp_N(Q,\nzero)$ which reduces $\#\Ind(\hua{P}_1-\hua{P}_2)$.
Thus we can iterated backward tilt $\h_2$ to $\h_1$ inductively,
which implies the lemma.
\end{proof}

Define the distance $\dis(\h_1, \h_2)$
and diameter $\dia(\h_1,\h_2)$ between $\h_1$ and $\h_2$
to be the minimum and respectively maximum over
the lengths of the paths in $\dpath(\h_1,\h_2)$.
Recall that we have the position function $\pf$ defined in Definition/Lemma~\ref{def:pf}.
Since $\tau^{h_Q}=[-2]$, we have
\[
    \pf(M[1])-\pf(M)=h_Q, \quad\forall M\in\AR(\D(Q)).
\]
Here $h_Q$ is the \emph{Coxeter number},
which equals $n+1, 2(n-1), 12, 18,30$ for $Q$ of type $A_n, D_n, E_6, E_7, E_8$ respectively.
There are the following easy estimations.

\begin{lemma}\label{lem:path}
Suppose that $\dpath(\h_1,\h_2)\neq0$.
Let $\hua{P}_i$ be the t-structure corresponding to $\h_i$.
We have
\begin{gather}
    \dia(\h_1,\h_2)\leq\#\Ind(\hua{P}_1-\hua{P}_2)\label{eq:di1}\\
    \dia(\h_1,\h_2)\leq\#\Ind(\hua{P}_2^\perp-\hua{P}_1^\perp)\label{eq:di1.5}\\
    \dis(\h_1,\h_2)\leq \frac{\pf(\h_2)-\pf(\h_1)}{h_Q}.\label{eq:di2}
\end{gather}
In particular $\dis(\h,\h[m])\geq nm$ with equality if $\h$ is standard.
\end{lemma}
\begin{proof}
For any edge $\h\to\tilt{\h}{\sharp}{S}$,
we have $\Ind\hua{P}\supsetneqq\Ind\tilt{\hua{P}}{\sharp}{S}$
and hence \eqref{eq:di1} follows. Similarly for \eqref{eq:di1.5}.

By \cite[Proposition~5.2]{Q} we have \cite[formula (5.2)]{Q}.
Notice that $T_j=\tilt{\psi}{\sharp}{S_i}(S_j)$ is a predecessor of $S_j$ and hence $\pf(T_j)<\pf(S_j)$.
We have
\[
    \pf( \tilt{\h}{\sharp}{S} )- \pf(\h)=\pf(S[1])-\pf(S)
    +\sum_{j\in\tilt{J}{\sharp}{i}}(\pf(T_j)-\pf(S_j))
    \leq\pf(S[1])-\pf(S)=h_Q,
\]
which implies the inequality \eqref{eq:di2}.

In particular, if $\h_1=\h$ and $\h_2=\h[m]$, the RHS of \eqref{eq:di2} equals $mn$.

Now suppose $\h$ is standard,  and without loss of generality let $\h=\nzero$.
Label the simples $S_1,...,S_n$ such that $\pf(S_1)\leq\pf(S_2)\leq...\leq\pf(S_n)$.
By Lemma~\ref{lem:homs}, $\Hom(M,L)\neq0$ implies $L$ is a successor of $M$
and hence $\pf(M)<\pf(L)$. Thus $\Hom^1(S_i, S_j)=0$ for $i>j$.
As in the proof of \cite[Corollary~5.3]{Q}, we can tilt from $\h$ to $\h[1]$
with respect to the simples $S_n , ... , S_1$ in order,
which implies $\dis(\h,\h[m])=mn$, as required.
\end{proof}

Next, we give a characterization of
the longest paths in $\dpath(\nzero,\nzero[1])$.

\begin{proposition}\label{pp:longest}
Let $\h$ be a standard heart,
then we have
\begin{gather}\label{eq:diam max}
    \dia(\h,\h[1])=\#\Ind\nzero=n\cdot \frac{h_Q}{2}.
\end{gather}
Moreover,
a path $p$ in $\dpath(\h,\h[1])$ has the longest length if and only if
all vertices of $p$ are standard hearts.
\end{proposition}
\begin{proof}
We can tilt from $\h$ to $\h[1]$ by a sequence of APR-tiltings,
which are L-tiltings (cf. Definition~\ref{def:L}).
By Corollary~\ref{cor:L-tilting}, such a path consisting of L-tiltings has length
\[
    \#\Ind(\hua{P}-\hua{P}[1])=\#\Ind(\hua{P}[1]^\perp-\hua{P}^\perp)
    =\#\Ind\nzero.
\]
Together with \eqref{eq:di1}, we have \eqref{eq:diam max}.

Suppose $p$ is a longest path.
We inductively show that any heart in $p$ is standard,
starting from the head $\nzero$ of $p$, which is standard.
Consider an edge $\h \to \tilt{\h}{\sharp}{S}$ in $p$
with $\h$ is standard.
Since $p$ is longest, by \eqref{eq:di1}, we have
\[\#\Ind(\hua{P}-\tilt{\hua{P}}{\sharp}{S})=1.\]
Noticing that $S\in(\hua{P}-\tilt{\hua{P}}{\sharp}{S})$, we have
\[
    \Ind \tilt{\hua{P}}{\sharp}{S}=\Ind\hua{P}-\{S\}.
\]
Similarly, we have
\[
    \Ind \left(\tilt{\hua{P}}{\sharp}{S}\right)^\perp=\Ind(\hua{P})^\perp \cup \{S\}.
\]
and hence
\begin{gather}\label{eq:ind}
    \Ind\hua{P} \cup \Ind\hua{P}^\perp=
    \Ind \tilt{\hua{P}}{\sharp}{S} \cup
    \Ind \left(\tilt{\hua{P}}{\sharp}{S}\right)^\perp.
\end{gather}
By Proposition~\ref{pp:standard},
a heart $\h'$ is standard if and only if
\[
    \Ind\hua{D}(Q) = \Ind\hua{P}' \cup \Ind(\hua{P}')^\perp.
\]
Therefore, by \eqref{eq:ind},
$\h$ is standard implies that so is $\tilt{\h}{\sharp}{S}$.
Thus the necessity follows.

On the other hand,
if $\h$ and its simple forward tilts $\tilt{\h}{\sharp}{S}$ are standard,
we claim that it is an APR-tilting at a sink.
Suppose not, that the vertex $V\in Q_0$ corresponding to $S$ is not a sink.
Then there is an edge $(V \to V') \in Q_1$ which
corresponds to a nonzero map in $\Ext^1(S,S')$,
where $S'$ is the simple corresponding to $V'$.
Then $S\notin(\tilt{\hua{P}}{\sharp}{S})^\perp$
since $S'[1]\in\hua{P}[1]\subset\tilt{\hua{P}}{\sharp}{S}$ by Lemma~\ref{lem:well known}.
Notice that $S\notin\tilt{\hua{P}}{\sharp}{S}$,
we know that $\tilt{\h}{\sharp}{S}$ is not standard by Proposition~\ref{pp:standard}~$2^\circ$,
which is a contradiction.
Thus if all the vertices of a path $p$ are standard
then it consists of only APR-tiltings, which are L-tiltings.
By Corollary~\ref{cor:L-tilting}, we know that
the length of $p$ is $\#\Ind\nzero$, which implies $p$ is longest by \eqref{eq:diam max}.
\end{proof}

\subsection{HN-strata}
In this subsection, we use Reineke's notion of HN-strata to
give an algebraic interpretation of
\[\dpath(Q)\colon=\dpath(\nzero,\nzero[1]).\]

\begin{definition}\label{def:HN}
A (discrete) \emph{HN-stratum} $\oset{T_l, ..., T_1}$
in an abelian category $\hua{C}$
is an ordered collection of objects $T_l,...,T_1$ in $\Ind\hua{C}$,
satisfying the HN-property:
\begin{itemize}
\item $\Hom(T_i,T_j)=0$ for $i>j$.
\item For any nonzero object $M$ in $\hua{C}$,
there is an HN-filtration by short exact sequences
\begin{gather}\label{eq:HN filt}
    \xymatrix@C=0.7pc{
    0=M_0 \ar[rr] && M_1 \ar[dl] \ar[rr] &&  ... \ar[rr] && M_{m-1} \ar[rr] && M_m=M \ar[dl] \\
    & A_{j_1} \ar@{-->}[ul]  && && && A_{j_m} \ar@{-->}[ul]
}\end{gather}
with $A_{j_i}$ is in $\<T_{j_i}\>$ and $1\leq j_m<...<j_1\leq l$.
\end{itemize}
\end{definition}
Notice that the uniqueness of HN-filtration follows from the first condition in HN-property.
Denote by $\HN(Q)$ the set of all HN-strata of $\nzero$.
We claim that there is a canonical bijection between $\dpath(Q)$ and $\HN(Q)$.

Let $p=T_l \cdot ... \cdot T_1$ be a path in $\dpath(Q)$
\begin{gather*}
    p: \nzero=\h_0 \xrightarrow{T_1}  \h_1 \xrightarrow{T_2}
     ...\xrightarrow{T_l} \h_l=\nzero[1]
\end{gather*}
with corresponding t-structures $\hua{P}_0\supset\hua{P}_1\supset...\supset\hua{P}_l$.
We have the following lemmas.

\begin{lemma}\label{lem:costruction}
For any indecomposable $M$ in $\nzero$, there is a filtration as \eqref{eq:HN filt}
such that $A_{j_i}$ is in $\<T_{j_i}\>$ and $1\leq j_m<...<j_1\leq l$.
\end{lemma}
\begin{proof}
We construct such a filtration as follows.
Since
\[
    M\in\hua{P}_0-\hua{P}_l=\bigcup_{i=1}^l \left(\hua{P}_{i-1}-\hua{P}_i\right),
\]
there exists an integer
$0<j\leq l$ such that $M\in\hua{P}_{j-1}-\hua{P}_{j}$.
Since $\h_{j}=(\h_{j-1})^\sharp_{T_{j}}$, we have a short exact sequence
\[\xymatrix{
    0\ar[r]& M' \ar[r]& M \ar[r]& A_j \ar[r]& 0
}\]
such that $A_j$ in $\<T_j\>$.
This is the last short exact sequence in the required filtration.
Since $M'$ is in the torsion part corresponding to $(\h_{j-1})^\sharp_{T_{j}}$,
we have
\[
    M' \in\hua{P}_j-\hua{P}_l=\bigcup_{i=j}^l \left(\hua{P}_{i-1}-\hua{P}_i\right).
\]
Therefore we can repeat the procedure above and the lemma follows by induction.
\end{proof}

\begin{lemma}\label{lem:tp1}
Suppose that $0\leq j\leq l$.
Let $\hua{F}_j=\<T_1, ..., T_j\>$ and $\hua{T}_j=\<T_{j+1},...T_l\>$.
Then $(\hua{F}_j,\hua{T}_j)$ is a torsion pair in $\nzero$
and $\h_{j}=\tilt{(\nzero)}{\sharp}{}$ with respect to this torsion pair.
\end{lemma}
\begin{proof}
Use induction on $j$ starting from the trivial case when $j=0$.
Now suppose that $\h_{j}=\tilt{(\nzero)}{\sharp}{}$ with respect to $(\hua{F}_j,\hua{T}_j)$.
Since $T_{j+1}$ is a simple in $\h_{j+1}$ and $T_k\in\hua{P}_{j+1}$ for $k>j+1$,
we have $\Hom(T_k, T_{j+1})=0$,
which implies $\Hom(A, B)=0$ for any $A\in\hua{T}_{j+1}, B\in\hua{F}_{j+1}$.
By Lemma~\ref{lem:costruction} we know that for any object $M$ in $\Ind\nzero$,
there is a short exact sequence $0\to A \to M \to B \to 0$  such that
$A\in\hua{T}_{j+1}$ and $B\in\hua{F}_{j+1}$.
Therefore $(\hua{F}_{j+1},\hua{T}_{j+1})$ is a torsion pair in $\nzero$.
By Lemma~\ref{lem:well known}, we have $\h_{j}\cap\nzero=\hua{T}_{j}$.
To finish we only need to show that $\h_{j+1}\cap\nzero=\hua{T}_{j+1}$.
This follows from $\h_{j+1}=\tilt{(\h_j)}{\sharp}{T_j}$.
\end{proof}

Now we have an injection $\dpath(Q)\to\HN(Q)$ as follows.

\begin{corollary}\label{cor:induce}
Any directed path $p=T_l \cdot ... \cdot T_1$ in $\dpath(Q)$
induces an HN-stratum $\oset{T_l, ..., T_1}$ in $\HN(Q)$.
\end{corollary}
\begin{proof}
Since $T_i\in\hua{F}_j$ and $T_j\in\hua{T}_j$ for $j>i$,
$\Hom(T_j,T_i)=0$ follows from Lemma~\ref{lem:tp1}.
Together with Lemma~\ref{lem:costruction}, the corollary follows.
\end{proof}

For the converse construction, we have the following lemma.

\begin{lemma}\label{lem:tp2}
Let $\oset{T_l, ... T_1}$ be an HN-stratum.
For $0\leq j\leq l$,
let $\hua{F}_j=\<T_1, ..., T_j\>$ and $\hua{T}_j=\<T_{j+1},...T_l\>$.
Then $(\hua{F}_j,\hua{T}_j)$ is a torsion pair in $\nzero$.
Let $\h_{j}=\tilt{(\nzero)}{\sharp}{}$ with respect to this torsion pair.
Then $T_{j+1}$ is a simple in $\h_j$ and $\h_{j+1}=\tilt{(\h_j)}{\sharp}{T_{j+1}}$.
\end{lemma}
\begin{proof}
Similar to the proof of Lemma~\ref{lem:tp1}.
\end{proof}

Combine the lemmas above, we have the following theorem.

\begin{theorem}\label{thm:HN}
The HN-strata in $\HN(Q)$ are precisely the directed paths in $\dpath(Q)$.
\end{theorem}

We will not distinguish $\HN(Q)$ and $\dpath(Q)$ from now on.

\begin{corollary}\label{cor:shortest}
For any shortest path $p$ in $\dpath(Q)$,
the set of labels of its edges are precisely $\Sim\nzero$.
\end{corollary}
\begin{proof}
The HN-filtration of a simple in $\nzero$ (with respect to $p$) can only
have one factor, i.e. itself.
Hence any simple of $\nzero$ appears in
an HN-stratum, and in particular, the labels of edges of any path $p$.
Thus the length of $p$ is at least $n$.
By Lemma~\ref{lem:path}, the length of a shortest path $p$ is exactly $n$
and hence the corollary follows.
\end{proof}

\subsection{Slicing interpretation}
Denote by $\sli(\D)$ the set of all slicing of a triangulated category $\D$
(cf. Definition~\ref{def:stab}).
We say a slicing $\hua{S}$ of $\D(Q)$ is discrete if
the abelian category $\hua{S}(\phi)$ is either zero or contains exactly one simple
for any $\phi\in\kong{R}$.
We say a heart $\h$ is in a slicing $\hua{S}$
if $\h=\hua{S}[\phi,\phi+1)$ or $\h=\hua{S}(\phi,\phi+1]$ for some $\phi\in\kong{R}$.
Let $\sli^*(\D(Q), \h)$ be the set of all discrete slicings of $\D(Q)$
that contain $\h$.

\begin{definition}
Let $\hua{S}_1$ and $\hua{S}_2$ in $\sli(\D)$.
If there is a monotonic (strictly) increasing function $\kong{R}\to\kong{R}$ such that
$\hua{S}_1(\phi)=\hua{S}_2(f(\phi))$,
then we say that the slicing $\hua{S}_1$ is homotopic ($\sim$) to $\hua{S}_2$.
\end{definition}

Now we can describe the relation between directed paths/HN-strata and slicings.

\begin{proposition}
There is a canonical bijection $\sli^*(\D(Q), \nzero)/_{\sim}\to\HN(Q)$.
\end{proposition}
\begin{proof}
Let $\hua{S}\in \sli^*(\D(Q), \nzero)$ and
suppose $\nzero=\hua{S}(I)$ for some interval $I$ with $|I|=1$.
Then it induces an HN-stratum by taking the collection of objects which are simple
in $\hua{S}(\phi)$ for $\phi\in I$ with decreasing order.
On the other hand,
an HN stratum $\oset{K_l, ..., K_1}$ is induced
by the slicing
\[
    \{\hua{P}(m+\frac{j}{l})=\<K_j[m]\> \mid j=1,..l \text{ and } m\in\kong{Z} \}.
\]
Hence we have a surjection $\sli^*(\D(Q), \nzero)\to\HN(Q)$ while the condition that
$\hua{S}_1$ and $\hua{S}_2$ map to the same HN-stratum
is exactly the homotopy equivalence.
\end{proof}

\subsection{Total stability}
Recall that we have the notion of a stability function on an abelian category
(Definition~\ref{def:sf}).
We call a stability function on $\hua{A}$ \emph{totally stable}
if every indecomposable is stable.
Reineke made the following conjecture.
\begin{conjecture}\cite{R2}\label{conj:R}
Let $Q$ be a Dynkin quiver.
There exists a totally stable stability function on $\nzero$.
\end{conjecture}
This was first proved by Hille-Juteau
(unpublished, see the comments after \cite[Corollary~1.7]{K6}).
We say a stability condition on a triangulated category
is \emph{totally stable} if any indecomposable is stable.
Let $\sigma=(Z,\hua{P})$ be a totally stable stability condition.
Then it will induce a totally stable stability function $Z$ on
any abelian category $\hua{P}(I)$,
for any half open half closed interval $I\subset\kong{R}$ with length $1$;
in particular, on its heart.
In the case of $\D(Q)$,
a totally stable stability condition
$\sigma$ induces a longest path in its heart $\h$ and forces $\h$ to be standard.
On the other hand, a totally stable stability function on any standard heart in $\D(Q)$
will induce a stability condition on $\D(Q)$,
which is also totally stable.

Now we give explicit examples to prove
the existence of a totally stable stability condition on $\D(Q)$,
which is a slightly weak version of Conjecture~\ref{conj:R}
because orientation matters.

\begin{proposition}\label{pp:R}
Let $Q$ be a Dynkin quiver.
There exists a totally stable stability condition on $\D(Q)$.
\end{proposition}
\begin{proof}
We treat the cases $A,D$ and $E$ separately.

For $A_n$-type, we use \cite[Example~A, \S~2]{R}.
Choose the orientation of $Q$ as
\[
  \xymatrix{  n \ar[r]& n-1 \ar[r]& \cdots \ar[r]& 1
}\]
Consider the stability function $Z$ on $\nzero$ defined by $Z(S_j)=-j+\mathbf{i}$.
Then $Z$ induces a totally stable stability condition on $\D(Q)$.

For $D_n$-type, choose the orientation of $Q$ as
\[\xymatrix@R=0.1pc{
        &&&& n-1 \\
        n-2 \ar@{<-}[r]& n-3 \ar@{<-}[r]& \cdots \ar@{<-}[r]& 1 \ar@{<-}[dr]\ar@{<-}[ur]\\
        &&&& n &}
\]
Consider the stability function $Z$ on $\nzero$ defined by
\begin{gather*}
 \left\{
  \begin{array}{l}
    Z(S_j)=j+\mathbf{i},\quad j=1,...,n-2,\\
    Z(S_{n-1})=Z(S_n)=t\mathbf{i},
  \end{array}
 \right.
\end{gather*}
where $t$ is a positive real number.
A simple calculation shows that $Z( \tau^{j} S_n)=Z( \tau^{j} S_{n-1})$ and
\begin{gather*}
    Z( \tau^{j-1} I_1)=2Z( \tau^{j} S_n)-{j}=j^2+(2t+2j-1)\mathbf{i},
    \quad 1\leq j\leq n-2.
\end{gather*}
Then the central change of all other indecomposables can be easily calculated.
So it is straightforward to check that
$Z$ induces a totally stable stability condition on $\D(Q)$
if $t\gg1$.

For the exceptional cases,
we use Keller's quiver mutation program \cite{Kqm}
to produce explicit examples of totally stable stability conditions for $E_{6,7,8}$.
Choose the orientation of $E_{6,7,8}$ as follows
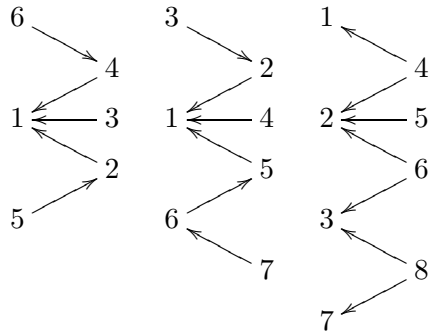
\[\xymatrix@R=.5pc{
    6\\ &4 \ar@{<-}[ul]\\ 1 \ar@{<-}[ur] \ar@{<-}[r] \ar@{<-}[dr] & 3 \\ &2 \ar@{<-}[dl] \\ 5
}
\quad
\xymatrix@R=.5pc{
    3\\ &2 \ar@{<-}[ul]\\ 1 \ar@{<-}[ur] \ar@{<-}[r] \ar@{<-}[dr] & 4 \\ &5 \ar@{<-}[dl] \\ 6 \ar@{<-}[dr] \\ & 7
}
\quad
\xymatrix@R=.5pc{
    1\ar@{<-}[dr]\\&4\\2\ar@{<-}[ur]\ar@{<-}[r]\ar@{<-}[dr]&5\\&6\\3\ar@{<-}[ur]\ar@{<-}[dr]\\&8\\7\ar@{<-}[ur]
}
\]
and we have following totally stable stability functions respectively:
\begin{equation*}
 \left\{
  \begin{array}{l}
    Z(S_1)= 258+9 \mathbf{i}\\
    Z(S_2)=-53 +32 \mathbf{i}\\
    Z(S_3)=-150 +36 \mathbf{i}\\
    Z(S_4)=-75 +33 \mathbf{i}\\
    Z(S_5)=-99 +64 \mathbf{i}\\
    Z(S_6)=-101 +10 \mathbf{i}
  \\
  \end{array}
 \right.
 \left\{
  \begin{array}{l}
    Z(S_1)=165 +10 \mathbf{i}\\
    Z(S_2)=-22 +33 \mathbf{i}\\
    Z(S_3)=-35 +36 \mathbf{i}\\
    Z(S_4)=-63 +37 \mathbf{i}\\
    Z(S_5)=-14 +28 \mathbf{i}\\
    Z(S_6)=-27 +21 \mathbf{i}\\
    Z(S_7)=-39 +24 \mathbf{i}
  \end{array}
 \right.
 \left\{
  \begin{array}{l}
    Z(S_1)= 47+ 16\mathbf{i}\\
    Z(S_2)= 135+ 9\mathbf{i}\\
    Z(S_3)= 93+ 11\mathbf{i}\\
    Z(S_4)= -66+ 40\mathbf{i}\\
    Z(S_5)= -57+ 32\mathbf{i}\\
    Z(S_6)= -92+ 57\mathbf{i}\\
    Z(S_7)= 42+ 25\mathbf{i}\\
    Z(S_8)= -45+ 45\mathbf{i}
  \end{array}
 \right.
\end{equation*}
where $S_i$ is the simple corresponding to vertex $i$.
Figure~\ref{pic:E6} is
the AR-quiver of the $E_{6}$ quiver under such a totally stable function, where
the bullets are simples/the origin and the stars are other indecomposables.
\begin{figure}\centering
\begin{xy} 0;<.8pt,0pt>:<0pt,-.8pt>::
(478,286) *+{\bullet} ="0",
(403,253) *+{\star} ="1",
(328,250) *+{\star} ="2",
(425,254) *+{\star} ="3",
(302,243) *+{\star} ="4",
(326,198) *+{\star} ="5",
(458,176) *+{\star} ="6",
(357,166) *+{\star} ="7",
(350,221) *+{\star} ="8",
(359,120) *+{\star} ="9",
(275,218) *+{\star} ="10",
(253,217) *+{\star} ="11",
(388,36) *+{\star} ="12",
(306,88) *+{\star} ="13",
(258,110) *+{\star} ="14",
(282,133) *+{\star} ="15",
(251,165) *+{\star} ="16",
(249,211) *+{\star} ="17",
(238,0) *+{\star} ="18",
(183,77) *+{\star} ="19",
(200,185) *+{\star} ="20",
(205,78) *+{\star} ="21",
(152,207) *+{\star} ="22",
(176,162) *+{\star} ="23",
(130,45) *+{\star} ="24",
(99,175) *+{\star} ="25",
(150,155) *+{\star} ="26",
(101,129) *+{\star} ="27",
(167,263) *+{\bullet} ="28",
(145,262) *+{\bullet} ="29",
(0,119) *+{\star} ="30",
(68,207) *+{\star} ="31",
(70,259) *+{\bullet} ="32",
(44,252) *+{\star} ="33",
(121,239) *+{\bullet} ="34",
(119,285) *+{\bullet} ="35",
(220,295) *+{\bullet} ="36",
(-10,295), {\ar(478,295)},
"0", {\ar"1"},
"0", {\ar"2"},
"0", {\ar"3"},
"1", {\ar"4"},
"1", {\ar"6"},
"2", {\ar"6"},
"3", {\ar"5"},
"3", {\ar"6"},
"4", {\ar"7"},
"5", {\ar"9"},
"6", {\ar"7"},
"6", {\ar"8"},
"6", {\ar"9"},
"7", {\ar"10"},
"7", {\ar"12"},
"8", {\ar"12"},
"9", {\ar"11"},
"9", {\ar"12"},
"10", {\ar"13"},
"11", {\ar"15"},
"12", {\ar"13"},
"12", {\ar"14"},
"12", {\ar"15"},
"13", {\ar"16"},
"13", {\ar"18"},
"14", {\ar"18"},
"15", {\ar"17"},
"15", {\ar"18"},
"16", {\ar"19"},
"17", {\ar"21"},
"18", {\ar"19"},
"18", {\ar"20"},
"18", {\ar"21"},
"19", {\ar"22"},
"19", {\ar"24"},
"20", {\ar"24"},
"21", {\ar"23"},
"21", {\ar"24"},
"22", {\ar"25"},
"23", {\ar"27"},
"24", {\ar"25"},
"24", {\ar"26"},
"24", {\ar"27"},
"25", {\ar"28"},
"25", {\ar"30"},
"26", {\ar"30"},
"27", {\ar"29"},
"27", {\ar"30"},
"28", {\ar"31"},
"29", {\ar"33"},
"30", {\ar"31"},
"30", {\ar"32"},
"30", {\ar"33"},
"31", {\ar"34"},
"33", {\ar"35"},
\end{xy}
  \caption{The AR-quiver of $\AR(\nzero)$ $E_6$-type
    under a totally stable stability function}\label{pic:E6}
\end{figure}
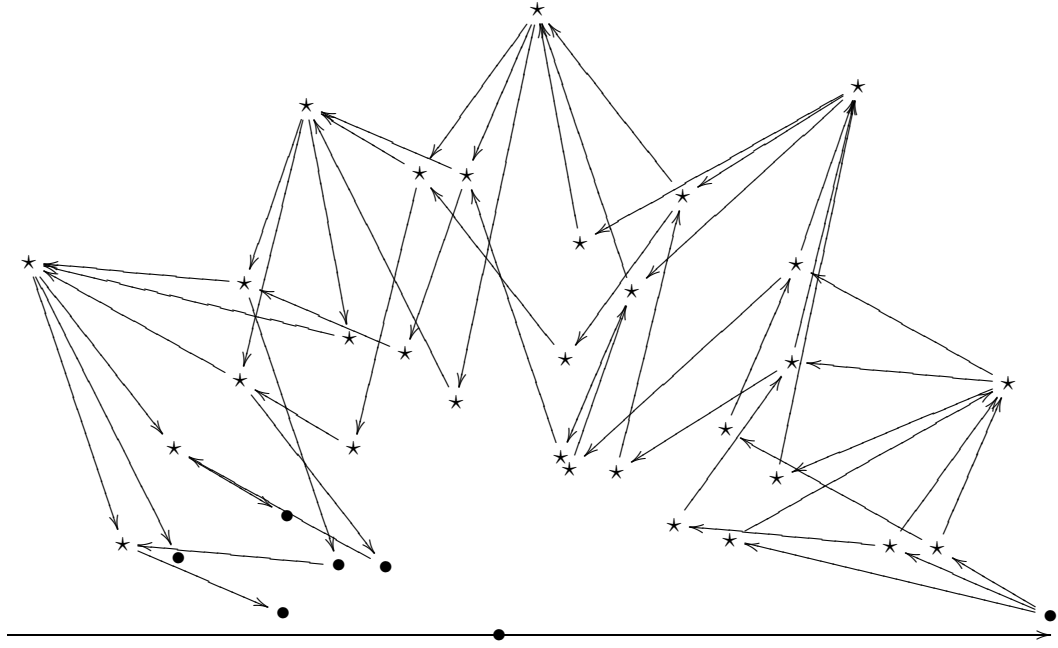

\end{proof}

\subsection{Inducing directed paths}\label{sub:FD}
We call a stability function \emph{discrete}, if $\mu_Z$ is injective
when restricted to the stable indecomposables.
\begin{proposition}\cite{King}\label{pp:King}
Let $Z:\K(\nzero)\to\kong{C}$ be a discrete stability function.
Then the collection of its stable indecomposables in the order of decreasing phase
is an HN-stratum of $\nzero$.
\end{proposition}

We say that a directed path in $\dpath(Q)$ is \emph{induced} if
the corresponding HN-stratum
is induced by some discrete stability function on $\nzero$.
Notice that, a totally stable stability function on $\nzero$
induces a directed path $p_s$ in $\dpath(Q)$ such that
there is an edge $M$ in $p_s$ for any $M\in\Ind\nzero$.
By \eqref{eq:diam max}, we know that $p_s$ is the longest path in $\dpath(Q)$.
Thus, in the language of exchange graphs,
Reineke's conjecture translates to, that
there exists a longest path in $\dpath(Q)$ which is induced.

It is natural to make a very strong generalization of Reineke's conjecture,
that any path in $\dpath(Q)$ is induced.
However, this is not true, even for some longest path shown below.

\begin{counterexample}\label{cex}
Let $Q$ be the following quiver of $D_4$-type
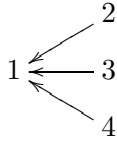
\[\xymatrix@R=.7pc{
    & 2 \ar[dl]\\ 1 & 3 \ar[l]\\ & 4 \ar[ul]
}\]
Then the AR-quiver of $\nzero$ is
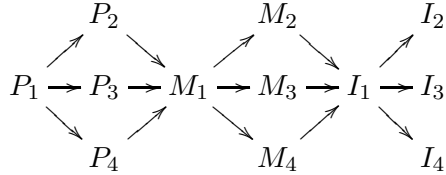
\[
 \xymatrix@R=1pc@C=1pc{
    & P_2 \ar[dr] && M_2 \ar[dr] && I_2\\
    P_1 \ar[ur]\ar[r]\ar[dr] & P_3 \ar[r] &
        M_1 \ar[ur]\ar[r]\ar[dr]& M_3 \ar[r] &
            I_1 \ar[ur]\ar[r]\ar[dr] & I_3\\
    & P_4 \ar[ur] && M_4 \ar[ur] && I_4\\
}\]
We claim that the following longest path
\begin{gather}\label{eq:counter}
    p=I_2 \cdot I_3 \cdot I_4 \cdot I_1 \cdot
        M_ 3\cdot M_4 \cdot M_2 \cdot M_1 \cdot
            P_2 \cdot P_3 \cdot P_4 \cdot P_1
\end{gather}
is not induced.
\begin{figure}[ht]\centering
\begin{tikzpicture}[scale=1,rotate=23]
\draw (0,0) node[below] {$0$};
\draw[dotted,->,>=latex] (157:6) -- (-23:4) node[right]{$x$};
\draw[dotted] (-5,4) -- (4.2,.1);

\path ( 0,8 ) coordinate (I1);
\draw[thick,->,>=latex] (0,0) -- (I1) node[above]{$Z(I_1)$};
\path ( -5,4 ) coordinate (I2);
\draw[thick,->,>=latex] (0,0) -- (I2) node[left]{$Z(I_2)$};
\path ( -1.35,2.05 ) coordinate (I3);
\draw[thick,->,>=latex] (0,0) -- (I3) node[left]{$_{Z(I_3)}$};
\path ( -1.8,5.4 ) coordinate (I4);
\draw[thick,->,>=latex] (0,0) -- (I4) node[left]{$^{Z(I_4)}$};
\path ( 5,4 ) coordinate (M2);
\draw[thick,->,>=latex] (0,0) -- (M2) node[right]{$Z(M_2)$};
\path ( 1.35,5.95 ) coordinate (M3);
\draw[thick,->,>=latex] (0,0) -- (M3) node[right]{$^{Z(M_3)}$};
\path ( 1.8,2.6 ) coordinate (M4);
\draw[thick,->,>=latex] (0,0) -- (M4) node[above right]{$_{Z(M_4)}$};
\path ( 2.7,3.9 ) coordinate (V);

\draw[dotted] (I2) -- (I1);
\draw[dotted] (I3) -- (I1);
\draw[dotted] (I4) -- (I1);
\draw[dotted] (M2) -- (I1);
\draw[dotted] (M3) -- (I1);
\draw[dotted] (M4) -- (I1);

\draw[dotted] (M3) -- (V) node[right] {$Z_V$};
\draw[dotted] (M4) -- (V);

\draw[thick,->,>=latex,cyan,dashed] (I2) -- (M3);
\draw[thick,->,>=latex,cyan,dashed] (I2) -- (V);
\draw[thick,->,>=latex,cyan,dashed] (I2) -- (M4);
\end{tikzpicture}
\caption{The central charges}\label{fig:parallelogram}
\end{figure}
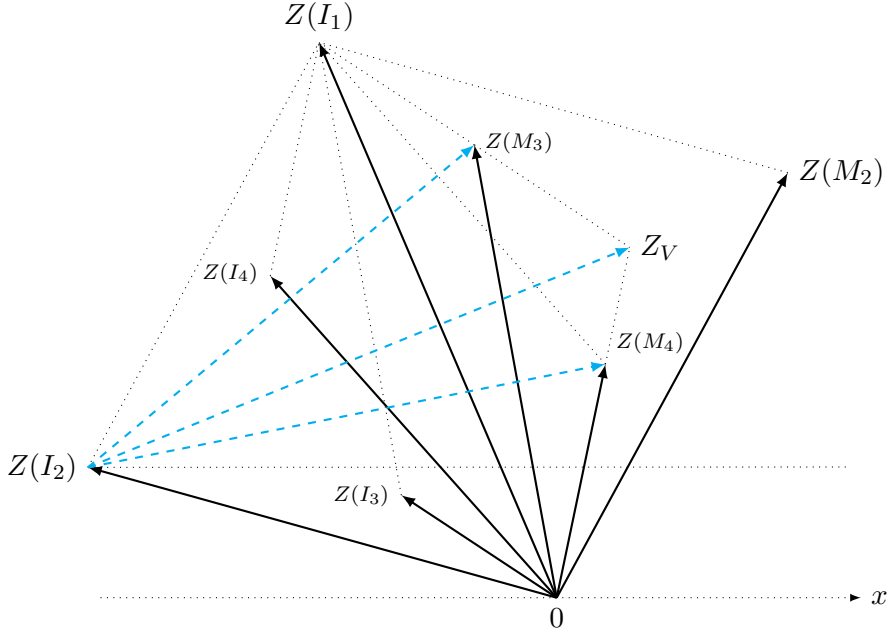
Suppose not, that $p$ is induced by some stability function $Z$.
The phase function $\mu_Z$ is decreasing on the edges in $p$
from left to right in \eqref{eq:counter}.
Then $Z(I_3),Z(I_4)$ are in the triangle with vertices $Z(I_2),Z(I_1)$ and origin $0$;
$Z(M_3),Z(M_4)$ are in the triangle with vertices $Z(I_1),Z(M_2)$ and $0$,
as shown in Figure~\ref{fig:parallelogram}.
Let $Z_V$ be the intersection of the line passing through points $Z(I_1),Z(M_3)$
and the line passing through points $Z(M_4),0$.
Noticing that
\[\mu_Z(P_3),\mu_Z(P_4)\in[0,\mu_Z(I_2)),\]
we have
\[
    \begin{array}{rll}
    \mu_Z(P_3) \pi&=\arg(Z(M_4)-Z(I_2))\\
              &<\arg(Z(M_4)-Z_V)\\
              &<\arg(Z(M_3)-Z(I_2))\\
              &=\mu_Z(P_4) \pi,
    \end{array}
\]
which is a contradiction.
\end{counterexample}

This suggests another generalization of Reineke's conjecture as follows.
We say two directed paths in $\dpath(Q)$ are \emph{weakly equivalent} if
the unordered sets of their edges coincide.

\begin{conjecture}
There is an induced path in each weak equivalence class in $\dpath(Q)$.
\end{conjecture}

Note that all longest paths in $\dpath(Q)$ form a weak equivalent class $E$.
Thus Reineke's conjecture can be stated as:
there is a path in the weak equivalence class $E$ which is induced.

\section{Quantum dilogarithms via exchange graph}\label{sec:QD via EG}
In this section, we define a quantum dilogarithm function on paths in exchange graphs,
which provides another proof of Reineke's identities (see Theorem~\ref{thm:R})
and the existence of DT-type invariants for a Dynkin quiver.

\subsection{DT-type invariant for a Dynkin quiver}\label{sec:DT.R}
Let $q^{1/2}$ be an indeterminate and
$\kong{A}_Q$ be the quantum affine space
\begin{gather}\label{eq:qas}
    \kong{Q}(q^{1/2})\big\{ y^\alpha \mid \alpha\in\kong{N}^{Q_0},
    y^\alpha y^\beta =q^{\frac{1}{2}(\<\beta,\alpha\>-\<\alpha,\beta\>)} y^{\alpha+\beta}
    \big\},
\end{gather}
where $\<-,-\>$ is the Euler form associated to $Q$
(see \S~\ref{sec:quiver.D}).
Denote $y^{\dim M}$ by $y^M$ for $M\in\nzero$.
Notice that $\kong{A}_Q$ can be also written as
\begin{gather*}
    \kong{Q}(q^{1/2})\< y^{S} \mid S\in\Sim\nzero \>  \big{/}
    (y^{S_i} y^{S_j} - q^{\lambda_Q(i,j)} y^{S_j} y^{S_i} ),
\end{gather*}
where $\lambda_Q(i,j)=\<S_j,S_i\>-\<S_i,S_j\>$.
Let $\widehat{\kong{A}}_Q$ be the completion of $\kong{A}_Q$
with respect to the ideal generated by $y^S, S\in\Sim\nzero$.

The DT-type invariant $\DT(Q)$ of the quiver $Q$ can be calculated
by the product \eqref{eq:E} as follows.
\begin{theorem}[Reineke \cite{R}, cf. \cite{K6}]\label{thm:R}
For any HN-stratum $\varsigma=\oset{K_l ,..., K_1}$ in $\HN(Q)$,
the product
\begin{gather}\label{eq:E}
    \DT(Q;\varsigma)\:=\kong{E}(y^{K_l})\cdot\kong{E}(y^{K_{l-1}})\cdots\kong{E}(y^{K_1})
\end{gather}
in $\kong{A}_Q$ is actually independent of $\varsigma$,
where $\kong{E}(X)$ is the quantum dilogarithm defined as the formal series
\[
    \kong{E}(X)=\sum_{j=0}^{\infty} \frac{q^{j^2/2}X^j}{
    \prod_{k=0}^{j-1}  (q^j-q^k)  }.
\]
\end{theorem}

In this subsection,
we will review Reineke's approach to Theorem~\ref{thm:R},
via identities in the Hall algebra (closely following \cite{K6}).

Let $\k_0$ be a finite field with $q_0=|\k_0|$
and consider the abelian category $\nzero(\k_0)=\mod \k_0 Q$.
Recall that the completed (non twisted, opposite) Hall algebra $\widehat{\hall}_{\k_0}(Q)$
consists of formal series with rational coefficients
\[\sum_{ [M]\in \nzero } a_m [M] ,\]
where the sum is over all isomorphism classes $[M]$ in $\nzero$.
The product in $\widehat{\hall}_{\k_0}(Q)$ is given by the formula
\[ [L][M]=\sum c_{LM}^K({q_0})[K] \]
where $c_{LM}^N(q_0)$ is the number of submodules $L'$ of $K$
such that $L'\cong L$ and $K/L'\cong M$ in $\nzero(\k_0)$.
Then the HN-property of an HN-stratum $\varsigma=\oset{K_l,...,K_1}$
translates into the identity (in Hall algebra) as
\begin{gather}\label{eq:MR}
    \sum_{ [M]\in \nzero } [M]= \prod_{j=1}^{l} \sum_{ [M]\in \<K_j\> } [M]
\end{gather}
Reineke showed that there is an algebra homomorphism (known as the \emph{integration})
\begin{gather*}
    \int\colon\widehat{\hall}_{\k_0}(Q) \to \widehat{\kong{A}}_{Q,q=q_0}\\
    [M] \mapsto q^{\<\dim M,\dim M\>}\frac{y^{M}}{|\Aut M|}.
\end{gather*}
By integrating \eqref{eq:MR},
a term $\sum_{ [M]\in \<K_j\> } [M]$ in the RHS gives $\kong{E}(y^{K_j})$
and hence the RHS gives $\DT(Q;\varsigma)$.
Notice that the LHS of \eqref{eq:MR} is clearly independent of $\varsigma$,
thus its integration gives the DT-type invariant $\DT(Q)$ for a Dynkin quiver $Q$.

\begin{example}\cite[Corollary~2.7]{K6}
By the proof of Lemma~\ref{lem:path},
we know that $\overrightarrow{\prod}_{S\in\Sim\h} S$ is
a shortest path in $\dpath(Q)$,
where the product is with respect to the increasing order of the position function
(if two objects have the same position function, then their order does not matter).
Moreover, by direct checking, we know that
$\overleftarrow{\prod}_{M\in\Ind\h} M$ is
a longest path in $\dpath(Q)$ consisting of APR tiltings,
where the product is with respect to the decreasing order of the position function.
Then these two paths (or the corresponding HN-strata) give the equality
\begin{gather}\label{eq:ls}
    \prod_{M\in\Ind\h}^{\longleftarrow} \kong{E}(y^{M})
    =\prod_{S\in\Sim\h}^{\longrightarrow} \kong{E}(y^{S}).
\end{gather}
\end{example}

\subsection{Generalized DT-type invariants for a Dynkin quiver}\label{sec:QDI}
We will give a combinatorial proof of Theorem~\ref{thm:R},
which provides a slightly stronger statement.

Let $p=\prod_{j=1}^l K_j^{\varepsilon_j}:\h \to \h'$ be
a path (not necessarily directed) in $\EG(Q;\nzero,\nzero[1])$,
where $K_j$ are edges in $\EG(Q)$ and
the sign $\varepsilon_j=\pm1$ indicates the direction of $K_j$ in $p$.
Define the quantum dilogarithm function of $p$ to be
\[
    \DT(Q;p)=\prod_{j=1}^{l} \kong{E}(y^{K_j})^{\varepsilon_i}.
\]
Since we identify HN-strata with directed paths in Theorem~\ref{thm:HN},
thus Theorem~\ref{thm:R} can be rephrased as:
the quantum dilogarithm of a directed path connecting $\nzero$ and $\nzero[1]$
is independent of the choice of the path.
It is natural to ask if the path-independence holds for
more general paths (not necessary directed).
The answer is yes within the subgraph $\EG(Q;\nzero,\nzero[1])$.

\begin{theorem}\label{thm:DT}
If $p$ is a path in $\EG(Q;\nzero,\nzero[1])$, then
$\DT(Q;p)$ only depends on the head $\h$ and tail $\h'$ of $p$.
\end{theorem}

\begin{proof}
By Proposition~\ref{pp:45},
$\pi_1(\EG(Q;\nzero,\nzero[1]))$ is generated
by the squares and pentagons as in \eqref{eq:45}.
Thus it is sufficient to check these two cases for the path-independence.

First, in either case, we have
we have $\Hom(S_i,S_j)=\Hom(S_j,S_i)=0$ and $S_i,S_j\in\nzero$.
For the square, we have
\[
    \Hom^1(S_i,S_j)=\Hom^1(S_j,S_i)=0
\]
and hence $\<\dim S_i, \dim S_j\>=\<\dim S_j, \dim S_i\>=0$
by \eqref{eq:euler form}, which implies
\[y^{S_i}\cdot y^{S_j}=y^{S_j} \cdot y^{S_i},\quad
    \kong{E}(y^{S_i})\cdot\kong{E}(y^{S_j})
    =\kong{E}(y^{S_j}) \cdot \kong{E}(y^{S_i})
\]
as required.
For the pentagon, we have
\[
    \Hom^1(S_i,S_j)=0, \quad \dim\Hom^1(S_j,S_i)=1
\]
and hence $\<\dim S_i$, $\dim S_j\>=0$ and $\<\dim S_j, \dim S_i\>=-1$
by \eqref{eq:euler form}.
By the relations of the quantum affine space we have
\begin{gather*}
    y^{S_i}\cdot y^{S_j}
    =q^{-1} \cdot y^{S_j} \cdot y^{S_i},\\
    y^{T_j}=q^{\frac{1}{2}} \cdot y^{S_i} \cdot y^{S_j},
\end{gather*}
noticing that $\dim T_j=\dim S_i+\dim S_j$ since $T_j$ is the extension of $S_i$ on top of $S_j$.
By the Pentagon Identity (see, e.g. \cite[Theorem~1.2]{K6}) we have
\begin{gather}\label{eq:ref-5}
    \kong{E}(y^{S_i})\cdot\kong{E}(y^{S_j})=
    \kong{E}(y^{S_j})\cdot\kong{E}(y^{T_j})\cdot\kong{E}(y^{S_i})
\end{gather}
as required.
\end{proof}

Therefore for any two hearts $\h_1,\h_2$ in $\EG(Q;\h,\h[1])$,
we have a \emph{generalized DT-type invariant}
\begin{gather}\label{eq:GDT}
    \DT(Q;\h_1,\h_2):=\DT(Q;p)
\end{gather}
where $p$ is any path connecting $\h$ and $\h'$.
In particular, we have
\begin{gather}\label{eq:DT}
    \DT(Q)=\DT(Q;\nzero,\nzero[1]).
\end{gather}

\subsection{Wall crossing formula for APR-tilting}
Let $i$ be a sink in $Q$ and $\Sim\h_{Q}=\{S_j\}_{j=1}^n$.
Then the APR-tilt $\h_{Q'}=\tilt{(\nzero)}{\sharp}{S_i}$ is also a standard heart in $D(Q)$,
where $Q'$ is obtained from $Q$ by reversing the arrows incident at $i$.
By \cite[Proposition~5.2]{Q}, we have $\Sim\h_{Q'}=\{T_j\}_{j=1}^n$,
where $T_i=S_i[1]$,
\[T_j=\Cone\left(S_j\to S_i[1]\otimes\Ext^1(S_j, S_i)^* \right) [-1]\]
for $j\neq i$.
Let $\dim'$ and $\<-,-\>'$ be the corresponding
dimension vector and the Euler form, respectively, associated to $Q'$.
Consider the quantum affine space $\kong{A}_{Q'}$
\begin{gather*}
    \kong{Q}(q^{1/2})\< z^T \mid T\in\Sim\h_{Q'} \>  \big{/}
    (z^{T_i} z^{T_j} =q^{\lambda_{Q'}(i,j)} z^{T_j} z^{T_i} )
\end{gather*}
where $z^T=z^{\dim' T}$ and $\lambda_{Q'}(i,j)=\<T_j,T_i\>'-\<T_i,T_j\>'$.
By Theorem~\ref{thm:DT},
we can also define DT-type invariants
$\DT(Q';\h_1,\h_2)$ in $\kong{A}_{Q'}$
for any $\h_1,\h_2\in\EG(Q;\h_{Q'},\h_{Q'}[1])$.

Notice that the labels of edges in $\EG(Q;\h_{Q'},\nzero[1])$ are in
\[\Ind(\nzero\cap\h_{Q'})=\Ind\nzero-\{S_i\}=\Ind\h_{Q'}-\{S_i[1]\}.\]
It is straightforward to check that the following conditions are equivalent
\numbers
\item for any hearts $\h_1,\h_2\in\EG(Q;\h_{Q'},\nzero[1])$,
    \[\DT(Q;\h_1,\h_2)=\DT(Q';\h_1,\h_2),\]
\item
    we have $z^{T_i}=(y^{S_i})^{-1}$ and $z^{M}=y^{M}$ for any
    $M\in\Ind(\nzero\cap\h_{Q'})$.
\item
    we have $z^{T_i}=(y^{S_i})^{-1}$ and $z^{T_j}=y^{T_j}$ for $j\neq i$.
\item
    we have $z^{T_i}=(y^{S_i})^{-1}$ and $z^{S_j}=y^{S_j}$ for $j\neq i$.
\ends
Further, if the conditions above hold,
the \emph{wall crossing formula}
\begin{gather}\label{eq:WC}
    \DT(Q)\cdot\kong{E}( y^{S_i} )^{-1}
    =\kong{E}( y^{-S_i} )^{-1}\cdot\DT(Q')
\end{gather}
comes for free because both sides are equal to
$\DT(Q;\h_{Q'},\nzero[1])$.

\begin{remark}
One can rephrase Keller's green mutation formula
(to calculate DT-invariants for quivers with potential)
as quantum dilogarithm functions on the corresponding exchange graphs in the same way,
cf. say \cite{E}.
In fact, exchange graphs are simplified version of
Keller's cluster groupoids in \cite{K6}.
\end{remark}

\appendix
\section{Connectedness of $\D(Q)$}\label{app}
We give another proof of the connectedness of the exchange graph for $\D(Q)$,
which was a result of Keller-Vossieck \cite{KV}.

\begin{definition}\label{def:L}
We say an indecomposable object $L$
in a subcategory $\hua{C} \subset \hua{D}(Q)$
is \emph{leftmost} if there is no path
from any other indecomposable in $\hua{C}$ to $L$,
or equivalently that no predecessor of $L$ is in $\hua{C}$.
In particular, a leftmost object in a heart is simple.
If in a simple forward tilting,
the simple object is leftmost, we call it a \emph{L-tilting}.
Similarly, an indecomposable object $R$ is \emph{rightmost}
if there is no path from $R$ to any other indecomposable in $\hua{C}$.
\end{definition}

\begin{lemma}
\label{lem:last}
Let $S$ be leftmost in $\h$
and $\h^\sharp=\tilt{\h}{\sharp}{S}$.
We have
\numbers
\item
$\big( \Ind \h \setminus \{S\}\big) \subset \h^\sharp$.
\item
Follow the notation of \cite[Proposition~3.3]{Q}.
If $m>1$, then $H_m^\hua{F}=0$.
\item
For any $M \in \Ind{\hua{D}(Q)}$,
$\Wid{\h^\sharp}M \leq \Wid{\h}M$.
\ends
\end{lemma}

\begin{proof}
Since $S$ is a leftmost object,
then $\Ind\hua{F}=\{S\}$.
For any indecomposable in $\h$ other than $S$,
we have $\Hom(M, S)=0$ which implies
$\big( \Ind \h \setminus \{S\}\big)
    \subset \hua{T} \subset \h^\sharp$.

For $2^\circ$, suppose $H_m^\hua{F}=S^j\neq0$,
then $M[-k_m]$ is the predecessor of $S$.
Consider an indecomposable summand $L$ of $H_1$.
If $L=S$, then $S[k_1]$ is the predecessor of $M$.
Since $k_1>k_m$, $S$ is the predecessor of $S[k_1-k_m]$, hence the predecessor of $M[-k_m]$.
Then $M[-k_m]$ and $S$ are predecessors to each other which is a contradiction.
If $L\neq S$, then $L \in \hua{T}$.
$L$ is the predecessor of $M[-k_1]$, hence the predecessor of $M[-k_m]$.
Then $L$ is the predecessor of $S$ which is also a contradiction.

For $3^\circ$,
if $\Wid{\h}M>0$, then $m>1$.
By $2^\circ$, $H_m^\hua{F}=0$.
Then by the filtration~(3.2) of \cite{Q},
$\Wid{\h^\sharp}M \leq k_1-k_m= \Wid{\h}M$.
If $\Wid{\h}M=0$, or equivalently $m=1$,
then by the filtration~(3.2) of \cite{Q} again,
$\Wid{\h^\sharp}M = 0 = \Wid{\h}M$.
\end{proof}

By the same argument in the proof of Lemma~\ref{lem:last} $2^\circ$,
we know that
an object $S$ is a leftmost object in a heart $\h$ in $\hua{D}(Q)$,
if and only if it is a leftmost object in the corresponding t-structure $\hua{P}$.

\begin{corollary}
\label{cor:L-tilting}
For a L-tilting with respect to a leftmost object $S$,
we have $\Ind\hua{P}^\sharp =\Ind\hua{P} -\{S\} $.
\end{corollary}
\begin{proof}
Consider the filtration~(3.1) of \cite{Q},
we have $M \notin \hua{P}$ if and only if $k_m < 0$.
If so,
since $H_m^\hua{T}$ or $H_m^\hua{F}$ is not $0$ in the filtration~(3.1) of \cite{Q},
then $M \notin \hua{P}^\sharp$.
Thus $\Ind\hua{P}^\sharp  \subset \Ind\hua{P}$.
On the other hand,
$M \in \Ind\hua{P}-\Ind\hua{P}^\sharp$
if and only if $H_m^\hua{F} \neq 0$ and $k_m=0$.
In which case, $m=1$ by Lemma~\ref{lem:last},
and hence $M=S$.
\end{proof}

\begin{lemma}
\label{lem:de}
For any object $M\in\Ind\hua{D}(Q)$,
if $\Wid{\h}M>0$,
then applying any sequence of L-tiltings to $\h$
must reduce the width of $M$ to zero after finitely many steps.
\end{lemma}

\begin{proof}
Suppose not,
let $\Wid{\h}M >0$ is the minimal width of $M$ under any L-tilting.
We have $m>1$ in filtration~(3.1) of \cite{Q}.
Then $H_m^\hua{F}=0$ by Lemma~\ref{lem:last}.
In the filtration~(3.2) of \cite{Q},
if $H_1^{\hua{T}}$ vanishes, then $\Wid{\h^\sharp}M \leq (k_1-1)-k_m< \Wid{\h}M$.
But $\Wid{\h}M $ is minimal, thus $H_1^{\hua{T}}\neq0$ after any L-tilting.

Consider the size of $H_1^{\hua{T}}$.
Let $H_1^{\hua{T}} = \oplus_{j=1}^l T_j^{s_j}$,
where $T_j$ are different indecomposables in $\hua{T}$ and $l$ is a positive integer.
By the filtration~(3.2) of \cite{Q} we know $H_1^{\hua{T}}$ will not change
if we do L-tilting that is not with respect to any $T_j$.
And if we do L-tilting with respect to some $T_j$,
then $H_1^{\hua{T}}$ will lose the summand $T_j$.
Since $H_1^{\hua{T}}$ cannot vanish,
we can assume after many L-tilting, $l$ reaches the minimum,
and we can not do L-tilting that is with respect to any $T_i$.

On the other hand, for any object $M\in \Ind\hua{D}(Q)$,
while $M[m]$ is the successor of some $T_j$ when $m$ is large enough,
it can not be leftmost in any heart that contains $T_j$.
Besides we can only do L-tilting with respect to any object once.
Thus, we will eventually have to tilt $T_j$,
which will reduce $l$ and it is a contradiction.
\end{proof}

Now we have a proposition about how one can do L-tilting.

\begin{proposition}\label{pp:L}
Applying any sequence of L-tiltings to any heart,
will make it standard after finitely many steps.
\end{proposition}

\begin{proof}
By Lemma~\ref{lem:de},
the width of any particular indecomposable will become zero
after finitely steps in the sequence.
But, up to shift, there are only finitely many
indecomposables in $\Ind\hua{D}(Q)$.
Thus, after finitely steps, we must reach a heart
with respect to which all indecomposables have width zero
and which is therefore standard, by Proposition~\ref{pp:standard}.
\end{proof}

Now we can prove the connectedness:

\begin{theorem}[Keller-Vossieck \cite{KV}]
\label{thm:conn}
$\EG(Q)$ is connected.
\end{theorem}
\begin{proof}
Since t-structure and heart are one-one correspondent,
any heart is connected to a standard heart by Proposition~\ref{pp:L}.
On the other hand,
using the equivalent definition $3^\circ$ in Proposition~\ref{pp:standard} for `standard',
the set of all standard hearts is connected by APR-tilting (cf.\cite[p~201]{ASS1}).
So the theorem follows.
\end{proof}


Mathematical Department, University of Bath, Bath, BA2 7AY, UK

E-mail address:  yu.qiu@bath.edu


\begin{thebibliography}{99}

\bibitem{ACV}
  M.~Alim, S.~Cecotti, C.~Cordova, S.~Espahbodi, A.~Rastogi and C.~Vafa,
  BPS Quivers and Spectra of Complete N=2 Quantum Field Theories,
  \href{http://arxiv.org/abs/1109.4941}{arXiv:1109.4941v1}.

\bibitem{ASS1}
  I.~Assem, D.~Simson and A.~Skowroski,
  Elements of the Representation Theory of Associative Algebras 1,
  \emph{Cambridge Uni. Press}, 2006.

\bibitem{Br-T}
  C.~Brav and H.~Thomas,
  Braid groups and Kleinian singularities,
  \emph{Math. Ann.}, 351(4)(2011), 1005-1017,
  (\href{http://arxiv.org/abs/0910.2521}{arXiv:0910.2521v3}).

\bibitem{B1}
  T.~Bridgeland,
  Stability conditions on triangulated categories,
  \emph{Ann. Math.} {166} (2007).
  (\href{http://arxiv.org/abs/math/0212237}{arXiv:math/0212237v3})

\bibitem{B2}
  T.~Bridgeland,
  Space of stability conditions,
  \href{http://arxiv.org/abs/math/0611510}{arXiv:math/0611510v1}.

\bibitem{BQS}
  T.~Bridgeland, Y.~Qiu, and T.~Sutherland.
  Stability conditions and the ${A}_2$ quiver.
  \href{http://arxiv.org/abs/1406.2566}{arXiv:1406.2566}.

\bibitem{C}
  L.~Conlon,
  Differentiable Manifolds,
  \emph{Birkh\"{a}user Boston}, 2nd ed., 2001.


\bibitem{F}
  W.~Fulton,
  Algebraic topology: A first course,
  \emph{Springer-Verlag} 1991.

\bibitem{E}
  M.~Engenhorst,
  Tilting and Refined Donaldson-Thomas Invariants,
  \href{http://arxiv.org/abs/1303.6014}{arXiv:1303.6014}.


\bibitem{G}
  V.~Ginzburg,
  Calabi-Yau algebras,
  \href{http://arxiv.org/abs/math/0612139}{arXiv:math/0612139v3}.

\bibitem{H}
  D.~Happel,
  On the derived category of a Finite-Dimensional algebra,
  \emph{Comment. Math. Helv}, 62(1987).

\bibitem{HRS}
  D.~Happel, I.~Reiten, and S.Smal\o,
  Tilting in abelian categories and quasitilted algebras,
  \emph{Mem. Amer. Math. Soc.} 120 (1996), no. 575, viii+ 88.

\bibitem{He}
  S.~Helgason,
  Differential Geometry, Lie Groups, and Symmetric Spaces,
  \emph{Academic Press}, 1978.




\bibitem{K6}
  B.~Keller,
  On cluster theory and quantum dilogarithm.
  \href{http://arxiv.org/abs/1102.4148}{arXiv:1102.4148v4}.


\bibitem{K10}
  B.~Keller,
  Cluster algebras and derived categories,
  \href{http://arxiv.org/abs/1202.4161}{arXiv:1202.4161v4}.

\bibitem{Kqm}
  B.~Keller,
  \href{http://www.math.jussieu.fr/~keller/quivermutation/}
  {Java program of quiver mutation}.

\bibitem{K4}
  B.~Keller and P.~Nicol\'{a}s
  Weight structures and simple dg modules for positive dg algebras.
  \emph{Int. Math. Res. Notices}, 2012.
  (\href{http://arxiv.org/abs/1009.5904}{arXiv:1009.5904v4})

\bibitem{KV}
  B.~Keller and D.~Vossieck,
  Aisles in derived categories,
  \emph{Bull. Soc. Math. Belg.} 40 (1988), 239-253.


\bibitem{KS}
  M.~Khovanov and P.~Seidel,
  Quivers, Floer cohomology and braid group actions,
  \emph{J. Amer. Math. Soc.} 15 (2002), 203--271.
  (\href{http://arxiv.org/abs/math/0006056}{arXiv:math/0006056v2})

\bibitem{King}
    A~King,
    Moduli of representations of finite-dimensional algebras, Quart. J. Math.
    \href{http://qjmath.oxfordjournals.org/content/45/4/515.citation}
    {\emph{Oxford Ser.} (2) 45 (1994), no. 180, 515-530}.

\bibitem{Q}
  A.~King and Y.~Qiu,
  Exchange graphs of acyclic Calabi-Yau categories,
  \href{http://arxiv.org/abs/1109.2924}{arXiv:1109.2924v2}.

\bibitem{Mu}
Y.~Mizuno,
  Classifying $\tau$-tilting modules over preprojective algebras
  of Dynkin type,
  \href{http://arxiv.org/pdf/1304.0667.pdf}{arxiv:1304.0667}.

\bibitem{N}
  K.~Nagao, Wall-crossing of the motivic Donaldson-Thomas invariants,
  \href{http://arxiv.org/abs/1103.2922v1}{arXiv:1103.2922v1}.

\bibitem{N1}
  K.~Nagao, Donaldson-Thomas theory and cluster algebras,
  \href{http://arxiv.org/abs/1002.4884}{arXiv:1002.4884v2}.

\bibitem{Q3}
  Y.~Qiu,
  C-sortable words as green mutation sequences,
  \href{http://arxiv.org/abs/1205.0034}{arXiv:1205.0034}.

\bibitem{Qiu-Woolf}
  Y.~Qiu and J.Woolf,
  Contracible stablilty mainfolds and faithful braid group actions,
  in preparation.

\bibitem{ST}
  P.~Seidel and R.~Thomas,
  Braid group actions on derived categories of coherent sheaves,
  \emph{Duke Math. J.}, 108(1):37-108, 2001.
  (\href{http://arxiv.org/abs/math/0001043}{arXiv:math/0001043v2})

\bibitem{R}
  M.~Reineke, The Harder-Narasimhan system in quantum groups and cohomology of quiver moduli, \emph{Invent. Math.} 152 (2003), no. 2, 349-368.
  (\href{http://arxiv.org/abs/math/0204059}{arXiv:math/0204059v1})

\bibitem{R2}
  M.~Reineke, Poisson automorphisms and quiver moduli,
  \emph{J. Inst. Math. Jussieu 9} (2010), no.3, 653-667.
  (\href{http://arxiv.org/abs/0804.3214}{arXiv:0804.3214v2})

\bibitem{S1}
  T.~Sutherland,
 The modular curve as the space of stability conditions of a CY3 algebra,
  \href{http://arxiv.org/abs/1111.4184}{arXiv:1111.4184v1}.

\bibitem{T1}
  R.~Thomas,
  Stability conditions and the braid group,
  \emph{Comm. Anal. Geom.}, 14(1), 135-161, 2006.
  (\href{http://arxiv.org/abs/math/0212214v5}{arXiv:math/0212214v5})

\bibitem{W}
  J.~Woolf,
  Stability conditions, torsion theories and tilting,
  \emph{J. Lon. Math. Soc.}, Vol 82, Issue 3, 663-682, 2010
  (\href{http://arxiv.org/abs/0909.0552}{arXiv:0909.0552v2}).

\end{thebibliography}
\end{document}